\documentclass[a4paper,11pt]{article}
\usepackage{verbatim}
\usepackage[T1]{fontenc}
\usepackage{lmodern}
\usepackage[latin1]{inputenc}
\usepackage{setspace}
\usepackage{indentfirst}
\usepackage{bbm}
\usepackage{amscd}

\usepackage{geometry}
\usepackage{hyperref}

\usepackage{times}
\usepackage{amsthm}
\usepackage{amsmath}
\usepackage{amsfonts}
\usepackage{amssymb}
\usepackage{bibentry}
\usepackage{color}
\usepackage{mathrsfs}
\usepackage{breqn}
\usepackage{stmaryrd}
\SetSymbolFont{stmry}{bold}{U}{stmry}{m}{n}

\usepackage{bm}

\usepackage{tikz-cd}
\usepackage{tikz}
\usetikzlibrary{matrix,arrows,decorations.pathmorphing}
\usepackage{graphicx}
\usepackage{subfigure}
\usepackage[toc,page]{appendix}

\DeclareMathOperator{\End}{End}

\newtheorem{Thm}{Theorem}[section]

\newtheorem{Pro}[Thm]{Proposition}
\newtheorem{Lem}[Thm]{Lemma}
\newtheorem{Cor}[Thm]{Corollary}
\newtheorem{Def-Pro}[Thm]{Definition-Proposition}
\newtheorem{Def}[Thm]{Definition}

\theoremstyle{definition}
\newtheorem{Ex}[Thm]{Example}
\newtheorem{Rm}[Thm]{Remark}

\begin{document}
\title{The Atiyah class of generalized holomorphic vector bundles}

\author{ \vspace{2mm}  Honglei Lang$^1$, \, Xiao Jia$^2$\, and \, Zhangju Liu$^3$ \\ \vspace{2mm}
\it{$^1$College of Science, China Agricultural University, Beijing, 100873, China}\\ \vspace{2mm}
\it{$^2$School of Mathematics, Peking University, Beijing, 100871, China}\\ \vspace{2mm}
\it{$^3$School of Mathematical and Statistical Sciences, Henan University, Kaifeng, 475004, China}\\ \vspace{2mm}hllang@cau.edu.cn,~~jiaxiao@pku.edu.cn,~~zhangju@henu.edu.cn}

\date{ }
\footnotetext{The research is supported by NSFC (12171006, 12071241).}
\maketitle

\makeatletter
\newif\if@borderstar
\def\bordermatrix{\@ifnextchar*{%
\@borderstartrue\@bordermatrix@i}{\@borderstarfalse\@bordermatrix@i*}%
}
\def\@bordermatrix@i*{\@ifnextchar[{\@bordermatrix@ii}{\@bordermatrix@ii[()]}}
\def\@bordermatrix@ii[#1]#2{%
\begingroup
\m@th\@tempdima8.75\p @\setbox\z@\vbox{%
\def\cr{\crcr\noalign{\kern 2\p@\global\let\cr\endline }}%
\ialign {$##$\hfil\kern 2\p@\kern\@tempdima & \thinspace %
\hfil $##$\hfil && \quad\hfil $##$\hfil\crcr\omit\strut %
\hfil\crcr\noalign{\kern -\baselineskip}#2\crcr\omit %
\strut\cr}}%
\setbox\tw@\vbox{\unvcopy\z@\global\setbox\@ne\lastbox}%
\setbox\tw@\hbox{\unhbox\@ne\unskip\global\setbox\@ne\lastbox}%
\setbox\tw@\hbox{%
$\kern\wd\@ne\kern -\@tempdima\left\@firstoftwo#1%
\if@borderstar\kern2pt\else\kern -\wd\@ne\fi%
\global\setbox\@ne\vbox{\box\@ne\if@borderstar\else\kern 2\p@\fi}%
\vcenter{\if@borderstar\else\kern -\ht\@ne\fi%
\unvbox\z@\kern-\if@borderstar2\fi\baselineskip}%
\if@borderstar\kern-2\@tempdima\kern2\p@\else\,\fi\right\@secondoftwo#1 $%
}\null \;\vbox{\kern\ht\@ne\box\tw@}%
\endgroup
}
\makeatother

\abstract
We introduce the notion of Atiyah class of a generalized holomorphic
vector bundle, which captures the obstruction to the existence of generalized holomorphic connections on the bundle.
As in the classical holomorphic case, this Atiyah class can be defined
in three different ways: using \v{C}ech cohomology,
using the first-jet short exact sequence,
or adopting the Lie pair point of view.

\tableofcontents
\section{Introduction}
The Atiyah class of a holomorphic vector bundle $E$ over a complex manifold $M$ was introduced by Atiyah in \cite{Atiyah}. It vanishes if and only if there exists a globally defined holomorphic connection on the bundle --- 
holomorphic connections always exist locally. Attempting to glue local holomorphic connections so as to obtain a global holomorphic connection gives rise to a 1-cocycle and then a cohomology class in the \v{C}ech cohomology group $\mathrm{H}^1(M,\Omega_M\otimes\End(E))$, where $\Omega_M$ is the holomorphic cotangent bundle of $M$; see \cite{H}. Equivalently,  the Atiyah class can be obtained by means of the first-jet bundle. Let $\mathfrak{J}^1 E$ be the vector bundle of first-jets of holomorphic sections of $E$. One gets  a short exact sequence 
\[0\to \Omega_M\otimes E\to \mathfrak{J}^1 E\to E\to 0,\]
of holomorphic vector bundles. The Atiyah class of $E$ is defined as the extension class of this exact sequence. This Atiyah class plays a crucial role in the Rozansky-Witten theory \cite{Kap99} and Kontsevich's work on deformation quantization \cite{K03}. Recently, in \cite{CSX}, Chen, Sti\'{e}non and Xu introduced the notion of the Atiyah class for a Lie pair, i.e., a Lie algebroid with a Lie subalgebroid. Besides the original situation, such Atiyah classes also include the Molino class for a regular foliation of a manifold, which measures the projectability of connections transversal to this foliation \cite{Mo71}, and the Atiyah class of a Lie algebra pair, which characterizes the existence of an invariant connection on a homogeneous space \cite{N}.
Moreover, the Atiyah class of commutative dg algebras  and the twisted version  were studied in \cite{CLX}.  
We also refer to \cite{CLangX, Hong, MSX} for other versions of Atiyah classes in different settings.

In this paper, we shall construct the Atiyah class of a generalized holomorphic vector bundle over a generalized complex manifold.
Generalized complex structures were discovered by Hitchin \cite{Hitchin} and further studied by Gualtieri \cite{Gualtieri}; see also \cite{Park}.   
Complex structures and symplectic structures are particular instances
of generalized complex structures.
Another important class of generalized complex structures is holomorphic Poisson structures.

Initially, generalized complex structures were defined as endomorphisms
of the generalized tangent bundle $\mathcal{T}:=TM\oplus T^*M$ rather
than the tangent bundle $TM$ of a manifold. The generalized tangent bundle
of a manifold is equipped with a canonical nondegenerate symmetric bilinear form and a canonical bracket called Courant bracket. 
A Dirac structure in $\mathcal{T}$ is a maximal isotropic integrable subbundle of $\mathcal{T}$.
A generalized complex structure is an endomorphism $\mathcal{J}:\mathcal{T}\to\mathcal{T}$ satisfying
$\mathcal{J}^2=-1$ whose $\pm i$-eigenbundles are Dirac structures
in $\mathcal{T}\otimes \mathbb{C}$.
Being a Dirac structure, the $-i$-eigenbundle $L_-$ of a generalized 
complex structure $\mathcal{J}$ is automatically a Lie algebroid. 
A vector bundle $E$ over a generalized complex manifold $(M,\mathcal{J})$ 
is called a generalized holomorphic vector bundle if it admits
a flat $L_-$-connection; see \cite{Gualtieri, Hitchin3}.
In particular, a generalized holomorphic vector bundle over a
\emph{complex} manifold $M$
is an ordinary holomorphic vector bundle $E\to M$
together with a holomorphic section of $T^{1,0}M\otimes\End(E)$,
i.e.\ a co-Higgs bundle \cite{R}.
On the other hand, a generalized holomorphic vector bundle over a
\emph{symplectic} manifold $M$ is a complex vector bundle $E\to M$
endowed with a flat connection.
Furthermore, a generalized holomorphic vector bundle over a holomorphic
Poisson manifold $M$ is a holomorphic vector bundle $E\to M$
carrying a Poisson module structure \cite{G,Hitchin4}. 

Notice that the total space of a generalized holomorphic vector bundle is not necessarily a generalized complex manifold. The obstruction was discussed in \cite{W1}. We would like to restrict our attention to the category of generalized complex manifolds. So we approach the notion of generalized holomorphic vector bundles from a geometric point of view
and propose a definition different from what can be found in the existing
literature --- see Definition~\ref{Definition3.1}. 
It turns out that our definition is a special case of Gualtieri's
--- see \cite{Gualtieri} ---
in the sense that the associated flat $L_-$-connection on $E$
is completely determined by the differential of the Lie algebroid $L_-$.
For simplicity, we use the same terminology for our definition. 
Furthermore, the generalized holomorphic tangent (cotangent) bundle
and the first-jet bundle are introduced, as examples. 
 
Following the idea of defining the Atiyah class of a holomorphic vector 
bundle, we introduce the Atiyah class of a generalized holomorphic
vector bundle, which captures the obstruction to the existence of a generalized holomorphic connection on it, in three different ways.
In this paper the Atiyah class is constructed only  for generalized holomorphic vector bundles of our definition, not Gualtieri's. 
To address the general case, one needs to define the Atiyah class
of an $L_-$-module relatively to the Manin pair $(TM\oplus T^*M,L_-)$.
Here a Manin pair is a Courant algebroid with a Dirac structure;  see \cite{LWX}. A tentative step has been considered in \cite{BP} to define the Atiyah class of Manin pairs.

Here is an outline of the present paper.
In Section $2$, we recall the essentials of generalized complex structures 
and clarify the definitions of generalized holomorphic maps and functions.
In section $3$, we propose a definition of generalized holomorphic vector 
bundles and provide several examples, including the first-jet bundle.
In Section $4$, we define the Atiyah class of a generalized holomorphic 
vector bundle in three different ways and prove that it is the obstruction 
to the existence of a generalized holomorphic connection. 

\section{Generalized complex manifolds}
\subsection{Generalized holomorphic functions}
 We first recall the generalized complex geometry following \cite{Gualtieri}.
 Let $M$ be a real $n$-dimensional smooth manifold. The direct sum $TM\oplus T^*M$ of its tangent and cotangent bundles  is equipped with a canonical bilinear form which takes values in $C^\infty(M)$:
 \begin{equation}\label{pairing}
( X+\xi,Y+\eta) =\xi(Y)+\eta(X),\qquad \forall X,Y\in \mathfrak{X}(M),\xi,\eta\in \Omega^1(M),
\end{equation} 
and a skew-symmetric bracket, called the \emph {Courant bracket}:
 \begin{eqnarray}\label{CB}
 [X+\xi,Y+\eta]:=[X,Y]+\mathcal{L}_X \eta-\mathcal{L}_Y \xi-\frac{1}{2}d(\eta(X)-\xi(Y)).
 \end{eqnarray}
The Jacobi identity of this bracket fails, but it is controlled by a coboundary:
\[[[e_1,e_2],e_3]+c.p.=\frac{1}{6}d(( [e_1,e_2],e_3)+( [e_2,e_3],e_1)+( [e_3,e_1],e_2)),\]
for $e_1,e_2,e_3\in \Gamma(TM\oplus T^*M)$, where $d$ is the de Rham differential on $M$.

This structure  leads to the definition of a Courant algebroid structure on a vector bundle, which was introduced and well studied in \cite{LWX}. It plays the role of the Manin triple for a Lie bialgebroid, which is a pair $(A,A^*)$ of Lie algebroids such that \[d_A[\xi,\eta]_*=[d_A \xi,\eta]_*+[\xi,d_A \eta]_*,\qquad \forall \xi,\eta\in \Gamma(A^*),\] where $d_A:\Gamma(\wedge^\bullet A^*)\to \Gamma(\wedge^{\bullet+1} A^*)$ is the Lie algebroid differential of $A$ and $[\cdot,\cdot]_*$ is the Lie bracket on $A^*$.
A \emph{ Dirac structure} of the Courant algebroid $TM\oplus T^*M$ is a maximal isotropic subbundle $A$ which is closed with respect to  Courant bracket \eqref{CB}. Obviously, a Dirac structure is a Lie algebroid. 
\begin{Def} \label{defgcg}$(\emph{\cite{Gualtieri}})$
A \emph{ generalized complex structure} on $M$ is an endomorphism $\mathcal{J}:TM\oplus T^*M \to TM\oplus T^*M$ satisfying that 
\begin{itemize}
\item[\rm(1)] $\mathcal{J}^2=-1$;
\item[\rm(2)] $(\mathcal{J}e_1,\mathcal{J}e_2)=(e_1,e_2)$, for $e_1,e_2\in \Gamma(TM\oplus T^*M)$;
\item[\rm(3)] $\mathcal{J}$ is integrable; namely, the $\pm i$-eigenbundle $L_\pm\subset (TM\oplus T^*M)\otimes \mathbbm{C}$ of $\mathcal{J}$ are closed under Courant bracket \eqref{CB}.
\end{itemize}
A manifold $M$ with a generalized complex structure $\mathcal{J}$ is called a \emph{ generalized complex manifold}.  
\end{Def}
A differential homeomorphism $f:(M,\mathcal{J}_M)\to (N,\mathcal{J}_N)$ between two generalized complex manifolds is called a  \emph{ generalized holomorphic homeomorphism}   if  \begin{equation}\label{ghh}
\begin{pmatrix} f_* & 0 \\0& (f^{-1})^* \end{pmatrix}\circ \mathcal{J}_M=\mathcal{J}_N\circ 
\begin{pmatrix} f_* & 0 \\0& (f^{-1})^* \end{pmatrix}.
\end{equation}
Since we have both $TM$ and $T^*M$ at the same time, it is quite involved to define generalized holomorphic maps; see  \cite{Crainic, Ornea}.

Note that $L_-=\overline{L_+}$. A  generalized complex structure is in fact a complex Dirac structure $L_+$ in $T_{\mathbbm{C}} M\oplus T_{\mathbbm{C}}^* M$ satisfying that $L_+\cap \overline{L_+}=\{0\}$.  
 It is seen that $L_+$ and $L_-$ are both Lie algebroids on $M$.  We then have two Lie algebroid differentials
 \begin{eqnarray}\label{d+-}
 d_+:\Gamma(\wedge^\bullet L_+^*)\to \Gamma(\wedge^{\bullet+1} L_+^*),\qquad d_{-}:\Gamma(\wedge^\bullet L_{-}^*)\to \Gamma(\wedge^{\bullet+1} L_{-}^*).
 \end{eqnarray}
 In particular, for $c\in \Gamma(\wedge^k L_-^*)$, we have 
 \begin{eqnarray*}
 d_-(c)(u_1,\cdots,u_{k+1})&=&\sum_{i=1}^{k+1}(-1)^{i+1}\rho(u_i)c(u_1,\cdots,\hat{u_i},\cdots,u_{k+1})\\ &&+\sum_{i<j}(-1)^{i+j} 
 c([u_i,u_j],u_1,\cdots,\hat{u_i},\cdots,\hat{u_j},\cdots,u_{k+1}),
 \end{eqnarray*}
 for $u_i\in \Gamma(L_-)$, where $\rho$ and $[\cdot,\cdot]$ are the anchor and Lie bracket of $L_-$.
 Note that $L_+^*\cong L_{-}$ with respect to the pairing \eqref{pairing} on $T_{\mathbbm{C}} M\oplus T_{\mathbbm{C}}^* M=L_+\oplus L_{-}$. By \cite[Theorem 2.6]{LWX}, two transversal Dirac structures of a Courant algebroid constitute a Lie bialgebroid. 
So the pair $(L_+,L_-)$ is a Lie bialgebroid. For any $f\in C^\infty(M)$, we have $df=d_+f+d_-f$, 
where  $d_+,d_-$ are defined in \eqref{d+-} and $d$ is the de Rham differential on $M$. Furthermore, there is  a Poisson structure on $M$ given by 
\begin{eqnarray}\label{Poisson}
\{f,g\}_M=(d_{+} f,d_{-} g),\qquad \forall f,g\in C^\infty(M),
\end{eqnarray}
where $(\cdot,\cdot)$ is the pairing in \eqref{pairing}. 
 
A  transformation $\mathcal{J}$ on $TM\oplus T^*M$ satisfying $\mathrm{(1)}$ and $\mathrm{(2)}$ in Definition \ref{defgcg} (or, $\mathcal{J}^*=-\mathcal{J}$) can be written  in block form as 
\begin{equation}\label{matrixj}
\mathcal{J}=\begin{pmatrix} J & \beta \\ B& -J^* \end{pmatrix},\qquad J\in \End(TM), B\in \Omega^2(M), \beta\in \mathfrak{X}^2(M).
\end{equation}
 A \emph{ $B$-field transformation}  (or, $B$-transform) is  defined to be an automorphism of $TM\oplus T^*M$ given by a $2$-form $B\in \Omega^2(M)$ via:  
 \[X+\xi\mapsto X+\xi+\iota_X B,\qquad \forall X\in \mathfrak{X}(M),\xi\in \Omega^1(M).\]
 One key characteristic of generalized complex geometry is that its symmetry group is given by the semi-direct product of the diffeomorphism group with closed $2$-forms: $\mathrm{Diff}(M)\ltimes\Omega^2_{\mathrm{cl}}(M)$.

\begin{Def}
A function $f\in C^\infty(M)$ on a generalized complex manifold $(M,\mathcal{J})$ is called a \emph{generalized holomorphic function} if  $d_{-} f=0$.
\end{Def}
\begin{Rm}\label{2remark}
\begin{itemize}
\item[\rm (i)] This definition is reasonable as the generalized holomorphic property of a function is invariant under the $B$-transform. Suppose $d_{-} f=X+\xi=0$. Under a $B$-transform $B\in \Omega^2(M)$, the Dirac structure $L_{-}$ is changed to $B(L_{-})$, and the corresponding differential becomes
\[d_{B(L_{-})} f=X+\xi+\iota_X B.\]
Therefore, $d_- f=0$ implies $d_{B(L_{-})} f=0$.  Observe that generalized holomorphic functions are Casimir functions for the Poisson bracket given by \eqref{Poisson}.
\item[\rm (ii)] Denote by $\mathcal{O}_M$ the ring of local generalized holomorphic functions on $M$. By definition, $\mathcal{O}_M=\ker d_-\subset C^\infty(M)$ is a subsheaf of the sheaf of smooth functions on $M$, since $d_-: C^\infty(U)\to 
\Gamma(L^*_-|_U)$, for any open subset $U\subset M$, is a presheaf homomorphism. 
\end{itemize}
\end{Rm}

 Complex and symplectic structures are two extreme cases of generalized complex structures. 
 \begin{Ex}\label{c}
  Let $J:TM\to TM$ be a bundle map. The endomorphism of $T_{\mathbbm{C}} M\oplus T_{\mathbbm{C}}^* M$ given by
\[\mathcal{J}_J=\begin{pmatrix} J & 0 \\0& -J^* \end{pmatrix}\]
is a generalized complex structure on $M$ if and only if $J$ is a complex structure on $M$. Actually, note that $\mathcal{J}_J^2=-1$ if and only if $J^2=-1$. Also, the  $\pm i$-eigenbundles with respect to $\mathcal{J}_J$ is \[(L_+)_J=T^{1,0} M\oplus (T^* M)^{0,1},\qquad (L_-)_J=T^{0,1} M\oplus (T^* M)^{1,0}\] which are integrable if and only if $J$ is integrable.  Moreover, the Lie algebroid differentials are respectively 
\[d_{-}=\overline{\partial}:\Gamma(\wedge^\bullet (L_+)_J)\to \Gamma(\wedge^{\bullet+1} (L_+)_J), \qquad d_+=\partial:\Gamma(\wedge^\bullet (L_-)_J)\to \Gamma(\wedge^{\bullet+1} (L_-)_J).\]
Thus $f\in C^\infty(M)$ is a generalized holomorphic function if and only if it is a holomorphic function. In this case, $\mathcal{O}_M$ is the sheaf of holomorphic functions.
 \end{Ex}
 \begin{Ex}\label{s}Let $\omega\in \Omega^2(M)$ be a $2$-form.
 Consider the endomorphism
\[\mathcal{J}_\omega=\begin{pmatrix} 0 & -\omega^{-1}\\ \omega & 0 \end{pmatrix}.\]
The $\pm i$-eigenbundles 
\[(L_+)_\omega=\{X-i\omega(X)| X\in T_{\mathbbm{C}} M\},\qquad (L_-)_\omega=\{X+i\omega(X)| X\in T_{\mathbbm{C}} M\}.\]
are integrable if and only if $d\omega=0$, where $d$ is the de Rham differential. Suppose that $\omega$ is symplectic. Then we have
\[d_{-}=d:\Gamma(\wedge^\bullet (L_+)_\omega)\to \Gamma(\wedge^{\bullet+1} (L_+)_\omega),\qquad d_+=0: \Gamma(\wedge^\bullet (L_-)_\omega)\to \Gamma(\wedge^{\bullet+1} (L_-)_\omega)\]
where  we have used the identifications $(L_+)_\omega \cong T^*M$ and $(L_-)_\omega \cong TM$. So a function $f\in C^\infty(M)$ is generalized holomorphic if it is a constant. In this case, $\mathcal{O}_M$ is the constant sheaf. \end{Ex}
 
Holomorphic Poisson structures provide another type of  generalized complex structures.  A \emph{ holomorphic Poisson structure} on a complex manifold $M$ is a bivector field $\pi\in \Gamma(\wedge^2 T^{1,0} M)$ such that \[\overline{\partial} \pi=0,\qquad [\pi,\pi]=0.\]
\begin{Ex}\label{holo Poisson}
Suppose that $M$ is a complex manifold with a complex structure $J\in \End(TM)$. If there is a bivector field $\beta\in \mathfrak{X}^2(M)$ on $M$ such that
\begin{equation*}
\mathcal{J}=\begin{pmatrix} J & \beta \\0& -J^* \end{pmatrix}\end{equation*}
is a generalized complex structure on $M$, then $\pi=J\circ \beta+i\beta: T^*_{\mathbbm{C}} M\to T_{\mathbbm{C}} M$
 gives a holomorphic Poisson structure on $M$; see \cite[Theorem 2.7]{LSX}.
Note that $\mathcal{J}^2=-1$ implies that $J\circ \beta=\beta\circ J^*$. 
The $\pm i$-eigenbundles corresponding to $\mathcal{J}$ are
\[L_+=\{Y+\frac{\beta(\eta)}{2i}+\eta|Y\in T^{1,0} M,\eta\in (T^*M)^{0,1}\},\qquad L_-=\overline{L_+}.\]
Notice that \[\pi^\sharp(\xi)=(J\circ \beta+i\beta)(\xi)=(\beta\circ J^*+i\beta)(\xi)=2i\beta(\xi),\qquad \forall \xi\in \Omega^{1,0}(M).\]
For a function $f:X\to \mathbbm{C}$, we have $d_{-} f\in \Gamma(L_{-}^*)$, which is given by
\[(d_{-}f, X+\frac{i\beta(\xi)}{2}+\xi)=(X+\frac{i\beta(\xi)}{2})(f)=X(\overline{\partial} f)-\frac{1}{4}\pi^\sharp(\partial f)(\xi),\]
for $X\in \mathfrak{X}^{0,1}(M), \xi\in \Omega^{1,0}(M)$.
This implies that \[d_{-}f=\overline{\partial} f-\frac{1}{4}\pi^\sharp(\partial f).\] Hence, $f$ is a generalized holomorphic function if and only if it is a holomorphic Casimir function. Furthermore, we have $d_+ f=\partial{f}+\frac{1}{4}\pi^\sharp(\partial f)$ so that the Poisson bracket \eqref{Poisson} is
\begin{eqnarray}\label{twopoisson}
\{f,g\}_M=(d_+f, d_-g)=\frac{1}{4}\pi(\partial f,\partial g),
\end{eqnarray}
as $\pi\in \Gamma(\wedge^2 T^{1,0} M)$ is  holomorphic.
\end{Ex}

Let $L_+$ be the $+i$-eigenbundle of the generalized complex structure and let $\rho: T_{\mathbbm{C}} M\oplus T^*_{\mathbbm{C}} M\to T_{\mathbbm{C}} M$ be the projection. Define $E=\rho(L_+)$ and we have $E\cap \overline{E}=\Delta\otimes \mathbbm{C}$.  We call $m$ a  regular point of $M$ if the real dimension of $\Delta$ is constant in a neighborhood of $m\in M$. If every point in $M$ is regular, we call $M$ a  regular generalized complex manifold. 
The type of a generalized complex structure at a point $m$ is defined as the type of the associated Dirac structure $(L_+)_m$, i.e., the codimension of $E_m\subset T_{\mathbbm{C}} M$. 

We have the  generalized Darboux Theorem given by Gualtieri in \cite{Gualtieri}: At a regular point $m$ of type $k$ of a generalized complex manifold $(M,\mathcal{J})$, there exists an open neighborhood $U_m$, such that up to a $B$-transform, $U_m$ is diffeomorphic to the product of an open subset $V$ of $\mathbbm{C}^k$ and an open subset $W$ of $(\mathbbm{R}^{2n-2k},\omega_0)$. Namely, there exists $B_\varphi\in \Omega_{\mathrm{cl}}^2(V\times W)$ such that the following generalized holomorphic homeomorphism 
\begin{eqnarray*}\label{localtri}
\varphi: (U_m,\mathcal{J})\cong (V\times W,e^{B_{\varphi}}(\mathcal{J}_{J_0}\times \mathcal{J}_{\omega_0} )e^{-B_{\varphi}})
\end{eqnarray*}
holds, where $J_0$ is the canonical complex structure on $\mathbbm{C}^k$ and $\omega_0$ is the standard symplectic structure on $\mathbbm{R}^{2n-2k}$.
Take a coordinate system $z=(z_1,\cdots,z_k)$  of $\mathbbm{C}^k$ and a coordinate system $(p,q)=(p_1,\cdots,p_{n-k},q_1,\cdots,q_{n-k})$ of $\mathbbm{R}^{2n-2k}$. We call the corresponding local coordinates 
\begin{eqnarray}\label{canon}
(U_m,\varphi; z,p,q)=(U_m,\varphi; z_1,\cdots,z_k, p_1,\cdots,p_{n-k}, q_1,\cdots,q_{n-k})
\end{eqnarray} the \emph{canonical coordinates}.

\begin{Pro}\label{coordinates2}
Let $(M,\mathcal{J})$ be a regular generalized complex manifold. Suppose that  $(U,\varphi;z,p,q)$ and $(U',\psi;z',p',q')$
are two canonical coordinates as in \eqref{canon} and $U\cap U'\neq \varnothing$. Then the coordinate transformation $\phi:=\psi\circ \varphi^{-1}:\varphi(U\cap U')\to \psi(U\cap U')$ 
 satisfies that
\[\frac{\partial z'_\lambda}{\partial \bar{z}_\mu}=0,\qquad \frac{\partial z'_\lambda}{\partial p_\nu}
=\frac{\partial z'_\lambda}{\partial q_\nu}=0,\qquad\lambda,\mu=1,\cdots,k; \nu=1,\cdots,n-k.\]
\end{Pro}
\begin{proof}
By definition, $\phi$ is a generalized holomorphic homeomorphism as defined by \eqref{ghh}. So we have
\begin{eqnarray*}\label{ghh1}
\begin{pmatrix} \phi_* & 0 \\0& (\phi^{-1})^* \end{pmatrix}\circ e^{B_\varphi} \circ \begin{pmatrix} J & \beta \\B& -J^* \end{pmatrix}\circ e^{-B_\varphi}=e^{B_\psi} \circ \begin{pmatrix} J & \beta \\B& -J^* \end{pmatrix}\circ e^{-B_\psi}
\circ \begin{pmatrix} \phi_* & 0 \\0& (\phi^{-1})^* \end{pmatrix},
\end{eqnarray*}
where 
\[J=\begin{pmatrix} J_0 & 0 \\0& 0 \end{pmatrix},\qquad \beta=\begin{pmatrix} 0 & 0 \\0& -\omega_0^{-1} \end{pmatrix},\qquad B=\begin{pmatrix} 0 & 0 \\0& \omega_0 \end{pmatrix}.\]
By substituting the $B$-transforms 
\[e^{B_\varphi}=\begin{pmatrix} 1 & 0 \\B_\varphi& 1 \end{pmatrix},\qquad e^{B_\psi}=\begin{pmatrix} 1 & 0 \\B_\psi& 1 \end{pmatrix}\]
 into the above relation, we obtain
\[\begin{pmatrix} \phi_* & 0 \\0& (\phi^{-1})^* \end{pmatrix}\circ \begin{pmatrix} J-\beta\circ B_\varphi & \beta \\ *& B_\varphi\circ \beta-J^* \end{pmatrix}=\begin{pmatrix} J-\beta\circ B_\psi & \beta \\ *& B_\psi\circ \beta-J^* \end{pmatrix}
\circ \begin{pmatrix} \phi_* & 0 \\0& (\phi^{-1})^* \end{pmatrix},\]
which implies that 
\begin{eqnarray}\label{blockmat1}
\phi_* \circ (J-\beta\circ B_\varphi)&=&(J-\beta\circ B_\psi)\circ  \phi_*,\\ 
\label{blockmat2}\qquad \phi_*\circ \beta&=&\beta\circ (\phi^{-1})^*.
\end{eqnarray}
Write $\phi_*$ in the matrix form $\phi_*:=(\phi_{ij})_{2\times 2}$ since $\varphi(U\cap U')$ is a product manifold. Taking into consideration of $J,\beta$ and $B$, \eqref{blockmat2} implies that $\phi_{12}=0$ and  \eqref{blockmat1} implies that $\phi_{11}\circ J_0=J_0\circ \phi_{11}$. In local coordinates, 
this means exactly that  $z'_\lambda (z,p,q)$ does not depend on the components $p$ and $q$ and it is holomorphic with respect to $z$. We get the desired result. 
\end{proof}

\begin{Ex}\label{coordinates}
Let $(M,\mathcal{J})$ be a regular generalized complex manifold. For any $m\in M$ of type $k$, take the canonical coordinates $(z,p,q)$ as in \eqref{canon}.
 A function $f:M\to \mathbbm{C}$ is a generalized holomorphic function if and only if 
\[\frac{\partial f}{\partial \bar{z}_\lambda}=0,\qquad \frac{\partial f}{\partial p_\mu}=\frac{\partial f}{\partial q_\mu}=0,\qquad \lambda=1,\cdots,k; \mu=1,\cdots,n-k.\]
\end{Ex}

In general, recent works  by Abouzaid, Boyarchenko (\cite{AB}) and Bailey (\cite{Bailey}) state that around an arbitrary point, a generalized complex manifold is equivalent to a product of a symplectic manifold and a holomorphic Poisson manifold.

\subsection{Generalized holomorphic maps}

In this subsection, we show that a function $f:M\to \mathbb{C}$ is generalized holomorphic if and only if it is a 
generalized homomorphic map in the sense of \cite{Ornea}.
Let us first recall linear generalized complex maps between generalized complex vector spaces.

Let $V$ be a vector space and $E\subset V$ be a subspace and let $\varepsilon\in \wedge^2 E^*$. Denote by
\[L(E,\varepsilon)=\{u+\alpha\in V\oplus V^*|u\in E,\alpha|_E:=\varepsilon(u)\in E^*\}.\] It is well-known that any maximal isotropic subspace of $V\oplus V^*$ takes the form of $L(E,\varepsilon)$.
Let $(V,\mathcal{J}_V)$ be a generalized complex vector space and let \[L_V=L(E_V,\varepsilon) \subset  V_{\mathbbm{C}} \oplus V^*_{\mathbbm{C}},\qquad E_V=\rho_{V_{\mathbbm{C}}} (L_V),\qquad \varepsilon\in \wedge^2 E_V^*\] be its $+i$-eigenspace, where $E_V$ is the projection of $L_V$ to $V_{\mathbbm{C}}$. Then we know that $L_V$ is a maximal isotropic subspace such that $L_V\cap \overline{L_V}=\{0\}$. This is equivalent to say that   $E_V+ \overline{E_V}=V\otimes \mathbbm{C}$ and the imaginary part $\omega_\Delta=\mathrm{Im}(\varepsilon|_{E_V\cap \overline{E_V}})$ of $\varepsilon$ restricting on $E_V\cap \overline{E_V}$ is nondegenarate on $E_V\cap \overline{E_V}=\Delta\otimes \mathbbm{C}$. So $P_V:=L(E_V\cap \overline{E_V},\omega_\Delta)$ is a linear Poisson structure on $V_{\mathbbm{C}}$.  A \emph{ linear Poisson structure}  on $V_{\mathbbm{C}}$ is defined as \[L(V^*,\pi)=\{\pi(\alpha)+\alpha|\alpha\in V_{\mathbbm{C}}^*,\pi\in \wedge^2 V_{\mathbbm{C}}\}.\]
 
\begin{Def}\rm{(\cite{Ornea})} A linear map $\phi: (V,L_V)\to (W,L_W)$ between two generalized complex vector spaces is called a  \emph{ generalized complex map} if
\[\phi(E_V)\subset E_W,\qquad \phi_*(P_V)=P_W,\]
where $\phi_*$ denotes the push-forward of a Dirac structure, namely,
\[\phi_*(P_V)=\{\phi(u)+\alpha\in W_{\mathbbm{C}}\oplus W_{\mathbbm{C}}^*|, u+\phi^*(\alpha)\in P_V\}.\]
\end{Def}
It is simple to see that the generalized complex map defined in this way is invariant under $B$-transforms and the composition of two generalized complex maps is still a generalized complex map.

\begin{Ex}
Suppose that $P_V=L(V^*,\beta_V)$ and $P_W=L(W^*,\beta_W)$ are two  linear Poisson structures on $V$ and $W$ respectively, where $\beta_V\in \wedge^2 V_{\mathbbm{C}}$ and $\beta_W\in \wedge^2 W_{\mathbbm{C}}$. The linear map $\phi:V\to W$ satisfies $\phi_*(P_V)=P_W$ if and only if  $\phi(\beta_V)=\beta_W$. \end{Ex}

A smooth map $f:(M,\mathcal{J}_M)\to (N,\mathcal{J}_N)$ between two generalized complex manifolds is called a \emph{ generalized holomorphic map} if $f_{*m}:T_m M\to T_{f(m)}N$ is generalized complex linear at every point $m\in M$.
In particular, one recovers holomorphic maps between complex manifolds and Poisson maps between symplectic manifolds.
\begin{Rm}
Notice that  if $f$ is a generalized holomorphic homeomorphism as defined by \eqref{ghh}, then $f$ and $f^{-1}$ are both generalized holomorphic maps. The converse is not true. In fact, the equation (\ref{ghh}) means that the generalized complex structures on $M$ and $N$ are the same under a differential homeomorphism. However, a generalized holomorphic map is defined up to  $B$-transforms. 
\end{Rm}

We always use $J_0$ to denote the standard generalized complex structure  on $\mathbbm{C}$.
\begin{Lem}\label{equi def}
Let $(V,\mathcal{J})$ be a generalized complex vector space with 
\begin{equation*}
\mathcal{J}=
\begin{pmatrix} J & \beta \\B& -J^* \end{pmatrix},\qquad J\in \End(V), \beta\in \wedge^2 V, B\in \wedge^2 V^*.\end{equation*}
Denote $\mathbbm{C}=(\mathbbm{R}^2,J_0)$. Then $f:V\to \mathbbm{C}$ is a generalized complex map if and only if 
\begin{eqnarray}\label{equi condition}
f\circ J=J_0\circ f,\qquad f\circ \beta=0.
\end{eqnarray}
\end{Lem}

\begin{proof}
Since the definition of generalized complex maps is $B$-transform invariant, we suppose $(V,\mathcal{J})=(V_1,\mathcal{J}_{J_1})\oplus (V_2,\mathcal{J}_{\omega_0})$, where $V_1=\mathbbm{R}^{2k}$, $V_2=\mathbbm{R}^{2n-2k}$ and $(V_1,J_1)=\mathbbm{C}^k$. Then we have 
\[E_V=V_1^{1,0}\oplus (V_2)_{\mathbbm{C}},\]
and the Poisson bivector on $V$ is $\beta_V=\omega_0^{-1}$.
Thus $f$ is a generalized complex map if and only if 
\[f(V_1^{1,0}\oplus (V_2)_{\mathbbm{C}})\subset \mathbbm{C}^{1,0},\qquad f(\omega_0^{-1})=0.\]
Namely,
\begin{eqnarray}\label{equi condition1}
f\circ J_1(v_1)=J_0\circ f(v_1),\qquad f(v_2)=0,\qquad \forall v_1\in V_1,v_2\in V_2.
\end{eqnarray}
Notice that $\omega_0$ is nondegenerate and 
\begin{equation*}
J=
\begin{pmatrix} J_1 & 0 \\0& 0 \end{pmatrix},\qquad \beta=\begin{pmatrix} 0 & 0 \\ 0 & -\omega_0^{-1} \end{pmatrix}.
\end{equation*}
The formulas (\ref{equi condition1}) are equivalent to (\ref{equi condition}).
\end{proof}

\begin{Pro}\label{Pro 2.29}
Let $(M,\mathcal{J})$ be a generalized complex manifold with $\mathcal{J}$ given by \eqref{matrixj}
and let $L_{-}$ be the $-i$-eigenbundle of $\mathcal{J}$.  For a function $f:U\subset M\to \mathbbm{C}$, the following statements are equivalent:
\begin{itemize}
\item[\rm{(1)}]  $f$ is a generalized holomorphic function;
\item[\rm{(2)}]  $f: U\subset M\to \mathbbm{C}=(\mathbbm{R}^2,J_0)$ is a generalized holomorphic map;
\item[\rm{(3)}] $f_*\circ J=J_0\circ f_{*},\qquad f_{*}\circ \beta=0$;
\item[\rm{(4)}] $(d f)_m\in (L_{-}\cap T_{\mathbbm{C}}^*M)_m$ for any $m\in U$.
\end{itemize}
\end{Pro}
\begin{proof}
By Lemma \ref{equi def}, we see the equivalence of $(2)$ and $(3)$. The equivalence of $(1)$ and $(4)$ is easy to get. In fact, by definition, if $f$ is a generalized holomorphic function, we have $d_{-}f=0$. Thus $d f =d_{+}f\in \Gamma(L_{-}\cap T_{\mathbbm{C}}^*M)$. Conversely, if $d f \in \Gamma(L_{-}\cap T_{\mathbbm{C}}^*M)$, we must have $d_{-}f=0$. 

It is left to check the equivalence of $(1)$ and $(3)$. If $(1)$ holds, i.e., $d_{-}f=0$, then $df=d_{+} f\in L_{-}$.  Namely, $\mathcal{J}(df)=-idf$, which implies that $\beta(df)=0$ and $J^*(df)=idf$. By the fact that $df=f^*(dz)$, we have
\[(J^*\circ f^*)(dz)=J^*(df)=if^*(dz)=(f^*\circ J_0^*)(dz).\]
As $df\in L_{-}$, we have $d\bar{f}\in L_+$. Similarly, we get $(J^*\circ f^*)(d\bar{z})=(f^*\circ J_0^*)(d\bar{z})$. This proves $f_*\circ J=J_0\circ f_*$. Following from $\beta(df)=0$, we have $\beta(d\bar{f})=0$. Thus we get $f_*\circ \beta=0$ and then $(3)$.

Conversely, suppose that $(3)$ holds. Let $d_{-} f=X_1+\xi_1$ and $d_{+} f=X_2+\xi_2$ for $X_1,X_2\in \Gamma(T_{\mathbbm{C}} M)$ and $\xi_1,\xi_2\in \Gamma(T_{\mathbbm{C}}^*M)$. Then we have
$X_1+X_2=0$ and $\xi_1+\xi_2=df$. By the fact that $d_{-} f\in L_+$ and $d_{+}f\in L_{-}$, we get
\[JX_1+\beta\xi_1=iX_1,\quad BX_1-J^* \xi_1=i\xi_1;\qquad J X_2+\beta \xi_2=-iX_2,\quad BX_2-J^*\xi_2=-i\xi_2.\]
Since $f_*\circ \beta=0$, we have $\beta(df)=0$. Adding the first and third equations, we derive $X_1=0$. Then the rest two equations tell us that $J^*\xi_1=-i\xi_1$ and $J^*\xi_2=i\xi_2$. By the relation $f_*\circ J=J_0\circ f_*$, we find
\[J^*(df)=(J^*\circ f^*)(dz)=(f^*\circ J_0^*)(dz)=if^*(dz)=idf.\]
That is, $J^*(\xi_1+\xi_2)=i(\xi_1+\xi_2)$. We thus get $\xi_1=0$. This proves that $d_{-}f=0$. Namely, $f$ is a generalized holomorphic function.
\end{proof}

\begin{Rm}
Another kind of generalized holomorphic maps  between two generalized complex manifolds  was introduced by Crainic  in \cite{Crainic} as follows:  A smooth map $f: (M_1,\mathcal{J}_1)\to (M_2,\mathcal{J}_2)$  is called a \emph{ generalized holomorphic map} if it satisfies
\[f_*(\beta_1)=\beta_2, \quad f^*B_2=B_1, \quad f_*\circ J_1=J_2\circ f_*,\quad where \quad \mathcal{J}_i=\begin{pmatrix} J_i& \beta_i \\B_i & -J_i^*\end{pmatrix},i=1,2.\]
This definition looks more natural, but it is not $B$-transform invariant. So here we  adopt the definition of generalized holomorphic  maps proposed in \cite{Ornea}, which is $B$-transform invariant.
\end{Rm}

\section{Generalized holomorphic vector bundles}

\subsection{Generalized holomorphic vector bundles}
We shall introduce a notion called  generalized holomorphic vector bundles in the same manner as for  holomorphic vector bundles by using the local trivialization.  Although this definition is stronger than the one given by Gualtieri in \cite{Gualtieri}, it has its own advantages as we shall see soon.
\begin{Def}\label{Definition3.1}
Suppose that  $M$ is a generalized complex manifold. A real vector bundle $\pi:E\to M$ is called a \emph{ generalized holomorphic vector bundle}, if 
\begin{itemize}
\item[\rm{(1)}] $E$ is a generalized complex manifold;
\item[\rm{(2)}] there is an open cover $\{U_i\}_{i\in I}$ of $M$ and a family of local trivializations of the vector bundle $E$ \[\{\varphi_i: E|_{U_i}=\pi^{-1}(U_i)\to U_i\times \mathbbm{C}^r\}_{i\in I}\] satisfying that
$\varphi_i$ for each $i$ is a  generalized holomorphic homeomorphism, where $U_i\times \mathbbm{C}^r$ is associated with the standard product generalized complex structure.
\end{itemize}
\end{Def}

\begin{Pro}\label{Explicit}
Let $E$ be a real vector bundle over $M$ with a family of local trivializations and the corresponding transition functions as follows:
\[\{\varphi_i:\pi^{-1}(U_i)\to U_i\times \mathbbm{R}^{2r}\},\qquad \varphi_{ij}=\varphi_i\circ \varphi_j^{-1}:U_i\cap U_j\to \mathrm{GL}(2r,\mathbbm{R}).\] Then $E$ is generalized holomorphic vector bundle over $M$ with local trivializations $\{\varphi_i\}$ 
 if and only if 
 \begin{itemize}
 \item[\rm{(1)}] $\varphi_{ij}(m)\in \mathrm{GL}(r,\mathbbm{C})$, namely, $E$ is a complex vector bundle;
 \item[\rm{(2)}] each entry $A_{\lambda\mu}: U_i\cap U_j\to \mathbbm{C}$ of $\varphi_{ij}=(A_{\lambda\mu})_{r\times r}$  is a generalized holomorphic function.
\end{itemize}
\end{Pro}
\begin{proof}
 First we claim that $E$ is a generalized holomorphic vector bundle with local trivializations $\{\varphi_i\}$ if and only if  \[\varphi_{ij}=\varphi_i\circ \varphi_j^{-1}:U_{ij}\times \mathbbm{C}^r\to U_{ij}\times \mathbbm{C}^r,\qquad U_{ij}=U_i\cap U_j\] is a generalized holomorphic homeomorphism for any fixed $i,j$. Indeed, if $\varphi_{ij}$ is a generalized holomorphic homeomorphism, we have
\begin{eqnarray}\label{3.7}
\begin{pmatrix} (\varphi_{ij})_* & 0\\ 0& (\varphi_{ji})^*\end{pmatrix}\circ 
\begin{pmatrix} \mathcal{J}_{11} & \mathcal{J}_{12}\\ \mathcal{J}_{21}& \mathcal{J}_{22}\end{pmatrix}=
\begin{pmatrix} \mathcal{J}_{11} & \mathcal{J}_{12}\\ \mathcal{J}_{21}& \mathcal{J}_{22}\end{pmatrix}\circ 
\begin{pmatrix} (\varphi_{ij})_* & 0\\ 0& (\varphi_{ji})^*\end{pmatrix},
\end{eqnarray}
where 
$(\mathcal{J}_{ij})_{2\times 2}$ is the generalized complex structure on $U_{ij}\times \mathbbm{C}^r$. 
Equation ($\ref{3.7}$) makes sure that 
\begin{eqnarray*}
\begin{pmatrix} (\varphi_i^{-1})_* & 0\\ 0& (\varphi_i)^*\end{pmatrix}\circ 
\begin{pmatrix} \mathcal{J}_{11} & \mathcal{J}_{12}\\ \mathcal{J}_{21}& \mathcal{J}_{22}\end{pmatrix}\circ 
\begin{pmatrix} (\varphi_i)_* & 0\\ 0& (\varphi_i^{-1})^*\end{pmatrix},\end{eqnarray*}
which maps $T\pi^{-1} (U_i)\oplus T^*\pi^{-1}(U_i)$ to itself, gives a generalized complex structure on $\pi^{-1}(U_i)$, which is independent of the choice of $\varphi_i$. Thus we obtain a generalized holomorphic vector bundle structure on $E$ such that $\varphi_i$ is a generalized holomorphic homeomorphism and $\pi:E\to M$ is a generalized holomorphic map. The other side of the claim is obvious. 

Then we claim that $\varphi_{ij}$ being a generalized holomorphic homeomorphism is equivalent to the conditions $(1)$ and $(2)$ as desired. 
Denote by $(V,J_0)=\mathbbm{C}^r$ and let $\mathcal{J}$ 
be the generalized complex structure on $M$ as given by \eqref{matrixj}. Then the generalized complex structure $\mathcal{J}$ on $U_{ij}\times V$ is 
expressed as \begin{equation*}
\mathcal{J}_{11}=\begin{pmatrix} J & 0\\ 0& J_0\end{pmatrix},\quad 
\mathcal{J}_{12}=\begin{pmatrix}\beta & 0\\ 0& 0\end{pmatrix},\quad
\mathcal{J}_{21}=\begin{pmatrix} B& 0\\ 0&0\end{pmatrix},\quad
\mathcal{J}_{22}=\begin{pmatrix} -J^* &0\\ 0 & -J_0^*\end{pmatrix}.
\end{equation*}
 Unraveling $(\ref{3.7})$ at $(m,v)\in U_{ij}\times V$, we get
 \begin{eqnarray}
 \label{unraveling1} (\varphi_{ij})_{*(m,v)}\circ \mathcal{J}_{11}&=&\mathcal{J}_{11}\circ  (\varphi_{ij})_{*(m,v)};\\
  \label{unraveling2} (\varphi_{ij})_{*(m,v)}\circ \mathcal{J}_{12}&=&\mathcal{J}_{12}\circ  (\varphi_{ji})^*_{(m,v)};\\
   \label{unraveling3} (\varphi_{ji})^*_{(m,v)}\circ \mathcal{J}_{21}&=&\mathcal{J}_{21}\circ  (\varphi_{ij})_{*(m,v)};\\
      \label{unraveling4} (\varphi_{ji})_{(m,v)}^*\circ \mathcal{J}_{22}&=&\mathcal{J}_{22}\circ  (\varphi_{ji})^*_{(m,v)}.
  \end{eqnarray}
  Notice that $\varphi_{ij}(m,v)=(m,\varphi_{ij}(m)(v))$ and $\varphi_{ij}(m)\in \mathrm{GL}(2r, \mathbbm{R})$.
  We have 
  \begin{eqnarray*}
  (\varphi_{ij})_{*(m,v)}=
\begin{pmatrix} \mathrm{Id} & 0\\ \iota_{v}\circ (\varphi_{ij})_{*m}& \varphi_{ij}(m)\end{pmatrix}:T_{m}U_{ij}\oplus T_{v} V\to T_m U_{ij}\oplus T_{\varphi_{ij}(m)v} V,
  \end{eqnarray*}
where $\iota_v: \mathrm{gl}   (2r, \mathbbm{R})\to \mathbbm{R}^{2r}$ is defined by $\iota_v(A)=A(v)$. Also we get \begin{eqnarray*}
  (\varphi_{ji})_{(m,v)}^*=
\begin{pmatrix}\mathrm{Id} & (\varphi_{ji})_m^*\circ \iota_{\varphi_{ij}(m)v}^*\\ 0& (\varphi_{ji}(m))^*\end{pmatrix}:T_{m}^*U_{ij}\oplus T_{v}^* V\to T_m^* U_{ij}\oplus T_{\varphi_{ij}(m)v}^* V.
  \end{eqnarray*}
Then from (\ref{unraveling1}),  we have
\begin{eqnarray}\label{3.12-13}
\varphi_{ij}(m)\circ J_0=J_0\circ \varphi_{ij}(m),\qquad \iota_{v}\circ (\varphi_{ij})_{*m}\circ J=J_0\circ \iota_v\circ (\varphi_{ij})_{*m}.
\end{eqnarray}
By (\ref{unraveling2}), we obtain
\begin{eqnarray*}
\iota_{v}\circ (\varphi_{ij})_{*m}\circ \beta=0,\qquad \beta\circ (\varphi_{ji})^*_m\circ \iota^*_{\varphi_{ij}(m)v}=0.
\end{eqnarray*}
Since $v$ is arbitrary and $\beta^*=-\beta$, the above two equations are equivalent and can be written as
\begin{eqnarray}\label{3.14}
(\varphi_{ij})_{*m}\circ \beta=0.
\end{eqnarray}
Note that Equation (\ref{unraveling3}) is trivial and (\ref{unraveling4}) holds if and only if   (\ref{unraveling1}) holds. So $\varphi_{ij}$ is a generalized holomorphic homeomorphism if and only if  (\ref{3.12-13}) and (\ref{3.14}) are satisfied, which amounts to that $\varphi_{ij}(m)\in \mathrm{GL}(r,\mathbbm{C})$ and \[(A_{\lambda \mu})_*\circ J=J_0\circ (A_{\lambda\mu})_*,\qquad (A_{\lambda\mu})_*\circ \beta=0.\]
By Proposition \ref{Pro 2.29}, the above relations imply that every $A_{\lambda\mu}$ is a generalized holomorphic function.
\end{proof}

\begin{Ex}
Let $\varphi: M\to N$ be a generalized holomorphic map and $E$  a generalized holomorphic vector bundle over $N$. Then the pull-back bundle $\varphi^*E\to M$ is also a generalized holomorphic vector bundle. Actually, the transition functions of $\varphi^* E$ are the pull-backs of the transition functions of $E$ by $\varphi$, i.e., composed by $\varphi$, which are still generalized holomorphic functions, as the composition of two generalized holomorphic maps is  still a generalized holomorphic map and by Proposition \ref{Pro 2.29}.
\end{Ex}
The standard operations on vector spaces, such as the direct sum, tensor product, exterior and symmetric power, etc., extend naturally to generalized holomorphic vector bundles.

By Proposition \ref{Explicit} and Example \ref{c} and \ref{s}, we have 

\begin{Cor}
\begin{itemize}
\item[\rm{(1)}] A generalized holomorphic vector bundle over a complex manifold is a holomorphic vector bundle;
\item[\rm{(2)}] A generalized holomorphic vector bundle over a symplectic manifold is a complex vector bundle with a flat connection.
\end{itemize}

\end{Cor}

A vector bundle $E$ on a Poisson manifold $(M,\pi)$ is called a \emph{ Poisson module} (\cite{G,Hitchin4}) if there is a bracket 
$\{\cdot,\cdot\}:C^\infty(M)\times \Gamma(E)\to \Gamma(E)$ such that
\begin{eqnarray*}
\{f,gs\}&=&\{f,g\}_M s+g\{f,s\};\qquad \forall f,g\in C^\infty(M),s\in \Gamma(E),\\  \{fg,s\}&=&f\{g,s\}+g\{f,s\};\\ \{\{f,g\}_M,s\}&=&\{f,\{g,s\}\}-\{g,\{f,s\}\},
\end{eqnarray*}
where $\{\cdot,\cdot\}_M$ is the Poisson bracket on $M$.
\begin{Pro}\label{newpro}
Let $E$ be a generalized holomorphic vector bundle over a generalized complex manifold $M$. Then $E$ is a Poisson module on the Poisson manifold $M$ with the Poisson bracket given by \eqref{Poisson}. Given any trivialization as in Definition \ref{Definition3.1} and the corresponding local frame $\{e_1,\cdots,e_r\}$ for $E$ on the open subset of $U_i$ of $M$, we have 
\begin{eqnarray}\label{poisson mods}
\{f,s\}:=\sum_{\lambda=1}^r\{f,s_\lambda\}_M e_\lambda,\qquad s|_{U_i}=\sum_{\lambda=1}^rs_{\lambda} e_\lambda.
\end{eqnarray}
\end{Pro}
\begin{proof}
 By Proposition \ref{Explicit}, $E$ is a generalized holomorphic vector bundle over $M$ if and only if its  transition functions are generalized holomorphic functions. We claim that the bracket $\{f,s\}$ defined in \eqref{poisson mods} does not depend on the choice of trivialization. Let $\{f_1,\cdots,f_r\}$ be another basis of sections of $E$ on $U_j$ given by $\varphi_j$  and $U_i\cap U_j\neq \varnothing$. If $\varphi_{ij}=(A_{\lambda\mu})_{r\times r}$ is the transition matrix, then $e_\lambda=\sum_{\mu=1}^r A_{\lambda\mu} f_\mu$. So 
\[s{|_{U_{ij}}}=\sum_{\lambda=1}^rs_\lambda e_\lambda=\sum_{\mu=1}^r (\sum_{\lambda=1}^r s_\lambda A_{\lambda \mu}) f_\mu.\]
Hence, we obtain
\[\{f,s\}:=\sum_{\mu=1}^r\sum_{\lambda=1}^r\{f,s_\lambda A_{\lambda\mu}\}_Mf_\mu=\sum_{\mu=1}^r\sum_{\lambda=1}^r A_{\lambda\mu}\{f,s_\lambda \}_M f_\mu=\sum_{\lambda=1}^r\{f,s_\lambda\}_M e_\lambda,\]
where we have used the fact that $A_{\lambda\mu}$ is Casimir with respect to $\{\cdot,\cdot\}_M$ as $d_- A_{\lambda \mu}=0$. So $\{f,s\}$ is well-defined. It is direct to check that it is a Poisson module based on the fact that $\{\cdot,\cdot\}_M$ is a Poisson bracket. 
 \end{proof}

\begin{Cor}Let $E$ be a complex vector bundle over a holomorphic Poisson manifold $(M,\pi)$. It is a generalized holomorphic vector bundle if and only if it is a holomorphic vector bundle with a Poisson module structure which is given by \eqref{poisson mods}.
\end{Cor}
\begin{proof}
On the one hand,  if $E$ is a generalized holomorphic  vector bundle, it has a Poisson module structure by Proposition \ref{newpro}. On the other hand, from the above proof, Formula \eqref{poisson mods} is well-defined if and only if each entry of  $A_{\lambda \mu}$ of  the transition matrix satisfies $\{f,A_{\lambda\mu}\}_M=0$ for any $f\in C^\infty(M)$, namely, it is Casimir with respect to $\pi$ by  \eqref{twopoisson}. As $E$ is a holomorphic vector bundle, we have that $A_{\lambda \mu}$ is a holomorphic function. So by  Example \ref{holo Poisson},  each $A_{\lambda\mu}$ is a generalized holomorphic function.  Thus  by Proposition \ref{Explicit}, $E$ is a generalized holomorphic vector bundle.
  \end{proof}

Let us recall the definition of an $A$-connection for a Lie algebroid $A$. In particular, an ordinary connection is a $TM$-connection.
\begin{Def}\label{3.6}
Let $A$ be a Lie algebroid. An \emph{ $A$-connection} on a vector bundle $E$ is an operator $\nabla: \Gamma(E)\to \Gamma(A^*\otimes E)$ such that
\[\nabla_{fu} e=f\nabla_u e,\qquad \nabla_u (fe)=\rho(u)f e+f\nabla_u e,\qquad \forall f\in C^\infty(M),u\in \Gamma(A),e\in \Gamma(E).\]
The \emph{ curvature} $R^\nabla=\nabla^2\in \Gamma(\wedge^2 A^*\otimes \End(E))$ of an $A$-connection $\nabla$  is defined by
\[R^\nabla(u,v)e=\nabla_u\nabla_v e-\nabla_v\nabla_u e-\nabla_{[u,v]} e,\qquad \forall u,v\in \Gamma(A), e\in \Gamma(E).\]
If $R^\nabla=0$, we say that $\nabla$ is a \emph{flat $A$-connection} and $E$ is an \emph{ $A$-module}. 
\end{Def}
With this definition, a vector bundle $E$ on a Poisson manifold $(M,\pi)$ is a Poisson module if it is a $T^*_\pi M$-module, where $T^*_\pi M$ is the Lie algebroid associated with the Poisson structure $\pi$ on $M$.

In complex geometry, for a complex vector bundle $E$ on a complex manifold $M$, a holomorphic structure on $E$ is uniquely determined by a $\mathbbm{C}$-linear operator $\overline{\partial}_E: \Gamma(U,E)\to \Gamma(U, (T^*M)^{0,1} \otimes E)$ satisfying the Leibniz rule and the integrability condition $\overline{\partial}_E^2=0$. See for example \cite{H}. 
Such an operator also exists for a generalized holomorphic vector bundle.
 
 \begin{Thm}\label{3.11}
For a generalized complex manifold $(M,\mathcal{J})$, let $L_{-}\subset T_{\mathbbm{C}} M\oplus T_{\mathbbm{C}}^* M$ be the $-i$-eigenbundle of $\mathcal{J}$, which is a Lie algebroid.  If $E$ is a generalized holomorphic vector bundle over $M$,  there exists an $L_{-}$-connection $\overline{\partial}_E$ on $E$ such that
$\overline{\partial}_E^2=0$.
\end{Thm}

\begin{proof}
Assume that $\{\varphi_i:E|_{U_i}\to U_i\times \mathbbm{C}^r\}$ is a local trivialization of $E$ and $\varphi_{ij}:U_{ij}\to \mathrm{GL}(r,\mathbbm{C})$ are the transition functions. Let $\{e_1,\cdots,e_r\}$ be a basis of smooth sections of $E$ on $U_i$ given by $\varphi_i$.  For any $s\in \Gamma(E)$, we have
\[s|_{U_i}=\sum_{\lambda=1}^r s_\lambda e_\lambda,\]
where $\{s_\lambda\}$ are complex  functions on $U_i$. Define 
\[\overline{\partial}_E(s)|_{U_i}:=\sum_{\lambda=1}^r d_{-}(s_\lambda)\otimes e_\lambda.\]
Following from $d_{-}^2=0$, we have $\overline{\partial}_E^2=0$. Now we check that $\overline{\partial}_E$ does not depend on the choice of $\varphi_i$.
Suppose $\{f_1,\cdots,f_r\}$ is another basis of sections of $E$ on $U_j$ given by the local trivialization $\varphi_j$. If $\varphi_{ij}=(A_{\lambda \mu})_{r\times r}$ is a transition matrix, then we have $e_\lambda=\sum_{\mu=1}^r A_{\lambda \mu} f_\mu$. So we have
\[s{|_{U_{ij}}}=\sum_{\lambda=1}^r\sum_{\mu=1}^r s_\lambda A_{\lambda \mu} f_\mu.\]
Hence, we obtain
\[\overline{\partial}_E (s)|_{U_{ij}}=\sum_{\lambda=1}^r\sum_{\mu=1}^r d_{-}(s_\lambda A_{\lambda \mu} )f_\mu
=\sum_{\lambda=1}^r\sum_{\mu=1}^r d_{-}(s_\lambda) A_{\lambda \mu} f_\mu=\sum_{\lambda=1}^r d_{-}(s_\lambda)\otimes e_\lambda,\]
where we have used the fact that $d_{-} (A_{\lambda \mu})=0$ as $A_{\lambda\mu}$ is a generalized holomorphic function. So $\overline{\partial}_E$ does not depend on the choice of local trivializations.
\end{proof}
\begin{Rm}In \cite{Gualtieri}, according to Gualtieri,  a generalized holomorphic vector bundle over a generalized complex manifold $M$ is defined as a vector bundle $E$ with an $L_{-}$-module structure for Lie algebroid $L_-$. As a consequence of Theorem \ref{3.11}, we see that our generalized holomorphic vector bundle  is a special case of the generalized holomorphic vector bundle by Gualtieri  with the module structure  given particularly by $d_{-}$. 

One advantage of our definition is that  the total space $E$ is still a generalized complex manifold. For Gualtieri's definition, this is not true. In \cite{W1}, the author gave the criterion for $E$ to be equipped with a generalized complex structure by using local coordinates, which turns out to be related with the Poisson structure \eqref{Poisson} on the base manifold $M$. The intrinsic conditions for the total space to be a generalized complex manifold are still unknown.
\end{Rm}

A section $s$ of a generalized holomorphic vector bundle $E$ is called \emph{ generalized holomorphic} if it satisfies that $\overline{\partial}_E s=0$. 
Choosing a local trivialization $\varphi: E|_U\to U\times \mathbbm{C}^r$, a section $s$ can be written locally as $s=(s_1,\cdots,s_r)$ for $s_i:U\to \mathbbm{C}$. By Proposition \ref{Pro 2.29}, we have  alternative descriptions of generalized holomorphic sections.

\begin{Lem}
Let $E\to M$ be a generalized holomorphic vector bundle and let $s$ be a section of the vector bundle $E\to M$. The following assertions are equivalent:
\begin{itemize}
\item[\rm{(1)}] $s:M\to E$ is a generalized holomorphic section;
\item[\rm{(2)}] In any local trivialization of $E$ on some open subset $U$ of $M$, every component $s_i: U\to \mathbb{C}$ of $s$ is a generalized holomorphic function;
\item[\rm{(3)}] $s:M\to E$ is a generalized holomorphic map. 
\end{itemize}
\end{Lem}
\begin{proof}
Let $\{e_1,\cdots,e_r\}$ be a basis of sections of $E$ on $U$ given by $\varphi$. Then $\overline{\partial}_E(s)=\sum_{\lambda=1}^rd_{-}(s_\lambda)\otimes e_\lambda$. So $\overline{\partial}_E(s)=0$ if and only if $d_{-}(s_\lambda)=0$. Namely, a section $s$ is a generalized holomorphic section if and only if all $s_\lambda$ are generalized holomorphic functions. Next, note that $s|_U:U\to U\times \mathbbm{C}^r, m\mapsto (m,s_1(m),\cdots,s_r(m))$. By Proposition \ref{Pro 2.29}, we see the equivalence of $(1)$, $(2)$ and $(3)$. \end{proof}

For similar reasons as in $\mathrm{(ii)}$ of Remark \ref{2remark}, the space of generalized holomorphic sections also forms a sheaf.

\begin{Def}\label{homomorphism}
A bundle map $f:E\to F$  between two generalized holomorphic vector bundles $E$ and $F$ is called a \emph{homomorphism} if  it is a generalized holomorphic map.
\end{Def}



\begin{Lem}
Let $E$ and $F$ be two generalized holomorphic vector bundles over $M$ and $f: E\to F$ a bundle map. Choose local trivializations $\varphi: E|_U\to U\times \mathbb{C}^k$ and $\phi: F|_V\to V\times \mathbb{C}^n$. Suppose $U\cap V\neq \varnothing$ and consider \[\hat{f}=\phi\circ f\circ  \varphi^{-1}: U\cap V\to \mathrm{gl}(n\times k,\mathbb{C}).\] Then $f$ is  a homomorphism between generalized holomorphic vector bundles if and only if  every entry of  $\hat{f}$ is a generalized holomorphic function.
\end{Lem}

\begin{proof}
Let $\mathcal{J}$
be the generalized complex structure on $M$ as given by \eqref{matrixj} and let $L_+$ be its $+i$-eigenbundle and $D=\rho_{T_{\mathbbm{C}} M}(L_+)$. For any point $(m,v)\in (U\cap V)\times \mathbb{C}^k$, we have
\begin{eqnarray*}
(\phi\circ f\circ \varphi^{-1})_{*(m,v)}=\begin{pmatrix} \mathrm{Id} & 0\\ \iota_{v}\circ \hat{f}_{*m}& \hat{f}(m)\end{pmatrix}: T_m (U\cap V)\oplus \mathbb{C}^k\to T_m (U\cap V)\oplus \mathbb{C}^n,
\end{eqnarray*}
where $\iota_v: \mathrm{gl}(n\times k,\mathbb{C})\to \mathbbm{C}^n$ is given by $\iota_v(A)=A(v)$. Then $f$ is a generalized holomorphic map if and only if $\phi\circ f\circ \varphi^{-1}$ is a generalized holomorphic map, if and only if
\[(\phi\circ f\circ \varphi^{-1})_{*(m,v)}(D_m\oplus (\mathbb{C}^k)^{1,0})\subset D_m\oplus (\mathbb{C}^n)^{1,0},\qquad (\phi\circ f\circ \varphi^{-1})_{*(m,v)}(\beta)=\beta.\]
This is equivalent to 
\[\iota_v \circ \hat{f}_{*m} (D_m)\subset (\mathbb{C}^n)^{1,0},\qquad \iota_v\circ \hat{f}_{*m}(\beta)=0,\]
which implies that each entry of $\hat{f}$  is a generalized holomorphic function.
\end{proof}

\subsection{The jet bundle}
Let $E$ be a generalized holomorphic vector bundle over a regular generalized complex manifold $M$. 
Recall that "regular" means that the dimension of $E\cap \bar{E}$  is locally constant, where $E=\rho(L_+)$.
We shall introduce the generalized holomorphic tangent and cotangent bundles of $M$, which are again generalized holomorphic vector bundles over $M$. Moreover, the  jet bundle $\mathfrak{J}^1 E$ is introduced and proved to be a generalized holomorphic vector bundle over $M$. The regularity assumption is actually not necessary if we use the sheaf language.

\begin{Pro}
Let $(M,\mathcal{J})$ be a regular generalized complex manifold and let $L_{-}\subset T_{\mathbbm{C}} M\oplus T^*_{\mathbbm{C}} M$ be the $-i$-eigenbundle of $\mathcal{J}$. Then $G^*M:=L_{-}\cap T^*_{\mathbbm{C}} M$ is a generalized holomorphic vector bundle over $M$, which is called the \emph{ generalized holomorphic cotangent bundle} of $M$.
\end{Pro}
\begin{proof}
Let $(U,\varphi; z,p,q)$ be a canonical coordinate system. Since $G^*M$ is invariant under $B$-transform, we have that the space of local sections of $G^*M$ is
\[\Gamma(G^* M{|_U})=\mathrm{span}_{C^\infty(U)}\{dz_1,\cdots,dz_k\}.\]
If $(V,\psi;z',p',q')$ is another canonical coordinate system and $U\cap V \neq \varnothing$, we have
\[\Gamma(G^* M{|_V})=\mathrm{span}_{C^\infty(V)}\{dz'_1,\cdots,dz'_k\}.\]
Following from Proposition \ref{coordinates2}, we get
\[dz'_\mu=\frac{\partial z'_\mu}{\partial z_\lambda} dz_\lambda+\frac{\partial z'_\mu}{\partial\bar{z}_\lambda} d\bar{z}_\lambda+
\frac{\partial z'_\mu}{\partial p_\nu} dp_\nu+\frac{\partial z'_\mu}{\partial q_\nu} dq_\nu=\frac{\partial z'_\mu}{\partial z_\lambda} dz_\lambda.\]
Note that \[\frac{\partial}{\partial \bar{z}_\alpha}(\frac{\partial z'_\mu}{\partial z_\lambda})=0,\qquad \frac{\partial}{\partial p_\beta}(\frac{\partial z'_\mu}{\partial z_\lambda})=\frac{\partial}{\partial q_\beta}(\frac{\partial z'_\mu}{\partial z_\lambda})=0,\qquad \alpha=1,\cdots,k, \beta=1,\cdots,n-k.\]
By Example \ref{coordinates}, we see that
the transition functions $\frac{\partial z'_\mu}{\partial z_\lambda}$ are generalized holomorphic functions. We thus conclude that $G^*M$ is a generalized holomorphic vector bundle.
\end{proof}
In the same manner, $GM:=L_+\cap T_{\mathbbm{C}} M$ is also a generalized holomorphic vector bundle over $M$ and we call it the \emph{ generalized holomorphic tangent bundle } of $M$. The generalized holomorphic sections of $G^*M$ and $GM$ are called generalized holomorphic $1$-forms and generalized holomorphic vector fields respectively.

\begin{Lem}\label{pairingnew}
The pairing of a generalized holomorphic $1$-form and a generalized holomorphic vector field is a generalized holomorphic function.
\end{Lem}
\begin{proof}
Observe that the spaces of generalized holomorphic $1$-forms and vector fields are
\[\mathrm{span}_{\mathcal{O}_M}\{dz_1,\cdots,dz_k\},\qquad \mathrm{span}_{\mathcal{O}_M}\{\frac{\partial}{\partial z_1},\cdots,\frac{\partial}{\partial z_k}\}.\]
The result is clear.\end{proof}

Consider the two particular cases in Examples \ref{c} and \ref{s}. When $M$ is a complex manifold,  we have $G^*M=(T^*M)^{1,0}$ and $GM=T^{1,0} M$, which are the holomorphic cotangent and tangent bundles of $M$, respectively. When $M$ is a symplectic manifold, then $G^*M$ and $GM$  both degenerate to a vector bundle of rank $0$ on $M$.

\begin{Ex}
For a  regular holomorphic Poisson manifold $(M,\mathcal{J})$, following Example \ref{holo Poisson}, the $-i$-eigenbundle of $\mathcal{J}$ is 
\[L_{-}=\{X+\frac{\pi^\sharp(\xi)}{4}+\xi|X\in T^{0,1} M,\xi\in (T^*M)^{1,0}\}.\]
Thus we get $G^*M=L_{-}\cap T^*_{\mathbbm{C}} M=\mathrm{ker} \pi^\sharp\cap (T^*M)^{1,0}$ and $GM=T^{1,0} M$. 
Locally, a generalized holomorphic $1$-form is of the form $\sum_i f_i(z)dz_i$ and a generalized holomorphic vector field is of the form $\sum_j g_j(z)\frac{\partial}{\partial z_j}$, where $f_i$ and $g_j$ are holomorphic Casimir functions on $M$.
\end{Ex}

\begin{Rm}
 Although $GM$ and $G^*M$ are not dual to each other as vector bundles, they are dual to each other as $\mathcal{O}_M$-sheaves of their generalized holomorphic sections, where $\mathcal{O}_M$ is the ring  of local generalized holomorphic functions on $M$. 
 \end{Rm}
 
 Denote by $\Gamma_m(E)$ the space of local generalized holomorphic sections around $m$ and $\mathcal{O}_m$ the space of local generalized holomorphic functions around $m$.  Two local sections $\phi,\psi\in \Gamma_m(E)$ are said to be \emph{ equivalent}, denoted by $\phi \sim \psi$, if and only if 
\[\phi(m)=\psi(m),\qquad \phi_{*m}=\psi_{*m}.\]
We denote the equivalence class of $\phi$ at $m$ as $[\phi]_m$, which is called the \emph{ (first) jet} of $\phi$ at $m$.
Define
\[\mathfrak{J}^1 E=\{[\phi]_m| m\in M,\phi\in \Gamma_m(E)\},\]
which we expect to be a generalized holomorphic vector bundle fitting into a short exact sequence. This could possibly be true when $M$ is a regular generalized complex manifold since the dimensions of fibers in $\mathfrak{J}^1 E$ may jump at points with different types of generalized complex structures.

By Proposition \ref{Pro 2.29},   
a function $f\in \mathcal{O}_m$ if and only if $(df)_m\in (L_-\cap T^*_{\mathbbm{C}} M)_m={G^*_m} M$. So  for a fixed $e_m\in E_m$, there is a one-one correspondence between the space of jets at $m$ 
\[\{[\phi]_m|\phi\in \Gamma_m(E),\phi(m)=e_m\}\]
and $G^*_m M\otimes E_m$.  Thus we get the short exact sequence of complex linear spaces
\[0\to G^*_m M\otimes E_m\to \mathfrak{J}_m^1 E\to E_m\to 0,\]
where $\mathfrak{J}^1_m E=\{[\phi]_m|\phi\in \Gamma_m(E)\}$.

Next, we shall prove that  the jet bundle $\mathfrak{J}^1 E$ is a generalized holomorphic vector bundle when $E$ is a generalized holomorphic vector bundle. Let us first give the description of the equivalence relation on generalized holomorphic sections in local coordinates. For $\phi,\psi\in \Gamma_m(E)$ with $\phi(m)=\psi(m)$, note that $\phi$ and $\psi$ are generalized holomorphic sections around $m$. We have $\phi\sim \psi$ if and only if  there exists a local coordinate system $(E|_{U_i},\varphi_i;z,p,q,u^\alpha)$ such that
\[\frac{\partial u^\alpha\circ \phi}{\partial z_{\lambda}}|_m=\frac{\partial u^\alpha\circ \psi}{\partial z_{\lambda}}|_m,\qquad \alpha=1,\cdots r; \lambda=1,\cdots,k,\]
where $\varphi_i:E|_{U_i}\to U_i\times \mathbbm{C}^r$ is a local trivialization of $E$, $(z,p,q)$ is a canonical coordinate system on $U_i$ and $(u^\alpha)_{\alpha=1}^r$ is a coordinate system along the fiber. 
By Lemma \ref{coordinates2}, this description of equivalence does not depend on the local trivialization we choose.

A coordinate system on $E$ will generate a coordinate system on $\mathfrak{J}^1 E$. 
For a  local coordinate system $(E|_{U_i},\varphi_i;z,p,q,u^\alpha)$ of $E$, the induced coordinate system on $\mathfrak{J}^1 E$ is \[(U_i^1,\varphi_i;z,p,q,u^\alpha,u_\lambda^\alpha),\] where
\[U_i^1=\{[\phi]_m|\phi(m)\in E|_{U_i},m\in U_i\}\]
and $u_\lambda^\alpha:U_i^1\to \mathbbm{C}$ is given by
\[u_\lambda^\alpha([\phi]_m)=\frac{\partial u^\alpha\circ \phi}{\partial z_{\lambda}}|_m.\]
This justifies that $\mathfrak{J}^1 E$ is a smooth manifold.

\begin{Thm} Let $E$ be a generalized holomorphic vector bundle over a regular generalized complex manifold $M$. Then $\mathfrak{J}^1 E$ is also a generalized holomorphic vector bundle over $M$ and it fits into the following short exact sequence
\begin{eqnarray}\label{ex se}
0\to G^*M\otimes E\xrightarrow{I} \mathfrak{J}^1 E\xrightarrow{\pi^1} E\to 0
\end{eqnarray}
of generalized holomorphic vector bundles over $M$.
\end{Thm}
\begin{proof}
To see that $\mathfrak{J}^1 E$ is a generalized holomorphic vector bundle,  it suffices to check  the two conditions in Proposition \ref{Explicit}.

Suppose that $(U_i^1,\varphi_i;z,p,q,u^\alpha,u_\lambda^\alpha)$ and $(U_j^1,\varphi_j;z',p',q',v^\beta,v_\mu^\beta)$ are two induced coordinate systems on $\mathfrak{J}^1 E$ and $U_i\cap U_j\neq \varnothing$. Since $E$ is a generalized holomorphic vector bundle, each entry $A_{\alpha \beta}:U_i\cap U_j\to \mathbbm{C}$ of its transition matrix $\varphi_{ij}=(A_{\alpha \beta})_{r\times r}$ is generalized holomorphic. Moreover, we have 
\begin{eqnarray}\label{coordinate change}
v^\beta=\sum_\alpha A_{\beta \alpha}(z,p,q) u^\alpha.
\end{eqnarray}
For any local generalized holomorphic section $\phi\in \Gamma_m(E)$, since
\[\frac{\partial v^\beta \circ \phi}{\partial \bar{z}_\lambda}=\frac{\partial v^\beta \circ \phi}{\partial p_\nu}=\frac{\partial v^\beta \circ \phi}{\partial q_\nu}=0,\]
we get
\[\frac{\partial v^\beta \circ \phi}{\partial z'_\mu}|_m=\sum_{\lambda} \frac{\partial v^\beta \circ \phi}{\partial z_\lambda}|_m \frac{\partial z_\lambda}{\partial z'_\mu}|_m=\sum_{\alpha,\lambda}\big( A_{\beta \alpha}|_m\frac{\partial u^\alpha \circ \phi}{\partial z_\lambda}|_m +\frac{\partial A_{\beta \alpha}}{\partial z_\lambda}|_m u^\alpha (\phi(m))\big)
\frac{\partial z_\lambda}{\partial z'_\mu}|_m.\]
Therefore, we get
\begin{eqnarray}\label{coordinate change2}
v_\mu^\beta=\sum_{\alpha,\lambda}\big( A_{\beta \alpha} u_\lambda^\alpha+\frac{\partial A_{\beta \alpha}}{\partial z_\lambda} u^\alpha \big)
\frac{\partial z_\lambda}{\partial z'_\mu}.
\end{eqnarray}
We see $(v^\beta,v_\mu^\beta)$ is complex linear with respect to $(u^\alpha,u_\lambda^\alpha)$ and the coefficients are generalized holomorphic functions. 
This proves that $\mathfrak{J}^1 E$ is a generalized holomorphic vector bundle over $M$.

To see the exactness of (\ref{ex se}), first it is obvious that the map
\[\pi^1: \mathfrak{J}^1 E\to E, \qquad [\phi]_m\mapsto \phi(m)\]
is a homomorphism between two generalized holomorphic vector bundles. 

Then we define the map $I: G^*M\otimes E\to \mathfrak{J}^1 E$ in a local coordinate system $(U_i^1,\varphi_i;z,p,q,u^\alpha,u^\alpha_\lambda)$ and show that it is independent of the choice of local trivializations. Let $\varphi_i:E|_{U_i}\to U_i\times \mathbbm{C}^r$ be a local trivialization of $E$ and let $\{e_\alpha\}$ be a basis of sections on $E$ given by $\varphi_i$ (i.e., $u^\alpha(e_\beta)=\delta_\beta^\alpha$). Define 
\[I:G^*M|_{U_i}\otimes E|_{U_i}\to \mathfrak{J}^1 E|_{U_i},\qquad \sum_{\lambda, \alpha} f_{\lambda \alpha} dz_\lambda \otimes e_\alpha \mapsto [\phi]_m,\quad \phi(m)=0_m,u_\lambda^\alpha([\phi]_m)=f_{\lambda \alpha}.\]
It is a homomorphism between generalized holomorphic vector bundles  from the local expressions.
Let $(U_j^1,\varphi_j;z',p',q',v^\beta,v_\mu^\beta)$ be another derivative coordinate system and let $\{e'_\beta\}$ be a basis of local sections on $E$ given by $\varphi_j$. By (\ref{coordinate change}), we have $e_\alpha=\sum_\beta A_{\beta \alpha} e'_\beta$. Then we have
\[\sum_{\lambda,\alpha} f_{\lambda \alpha} d z_\lambda\otimes e_\alpha=\sum_{\lambda,\alpha,\mu,\beta}f_{\lambda\alpha}\frac{\partial z_\lambda}{\partial z'_\mu} A_{\beta \alpha} d z'_\mu\otimes e'_\beta.\]
Let $\phi'\in \Gamma_m(E)$ such that $\phi'(m)=0_m$ and \[v_\mu^\beta([\phi']_m)=\sum_{\lambda,\alpha}f_{\lambda\alpha}\frac{\partial z_\lambda}{\partial z'_\mu} A_{\beta \alpha}.\]
To prove $I$ is independent of the choice of local trivializations, we need to check $v_\mu^\beta([\phi]_m)=v_\mu^\beta([\phi']_m)$. Actually, by (\ref{coordinate change2}) and $u^\alpha([\phi]_m)=0$, we get
\[v_\mu^\beta([\phi]_m)=\sum_{\alpha,\lambda} A_{\beta \alpha} u^\alpha_\lambda([\phi]_m) \frac{\partial z_\lambda}{\partial z'_\mu}=\sum_{\alpha,\lambda} A_{\beta \alpha} f_{\lambda \alpha} \frac{\partial z_\lambda}{\partial z'_\mu}=v_\mu^\beta([\phi']_m).\]
So $I$ is well-defined. It is direct to see $\mathrm{img} I=\mathrm{ker} \pi^1$.
\end{proof}

\section{Atiyah classes}
We shall introduce the notion  of generalized holomorphic connections on a generalized holomorphic vector bundle. The obstruction class for the existence of such a connection is called the \emph{ Atiyah class} of this generalized holomorphic vector bundle.
\begin{Def}
Let $E$ be a generalized holomorphic vector bundle over  $(M,\mathcal{J})$. A \emph{ generalized holomorphic connection} on $E$ is a $\mathbbm{C}$-linear map (of sheaves) $D:E\to G^*M\otimes E$ such that
\[D(fs)=d_{+}f\otimes s+fD(s)\]
for all local generalized holomorphic functions $f$ on $M$ and all local generalized holomorphic sections $s$ of $E$, where $d_+$ is the Lie algebroid differential of the $+i$-eigenbundle  $L_+$ of $\mathcal{J}$ as in \eqref{d+-}.
\end{Def}

Here, we use $E$ and $G^*M$ to denote both the vector bundles and the sheaves of their generalized holomorphic sections. Observe that $d_{-} (d_{+} f)=d_{+}(d_{-}f)=0$, for any local generalized holomorphic function $f$ on $M$. So $d_{+}f$ is a generalized holomorphic section of $G^*M$.

\begin{Lem}\label{perserve}
Let $E$ be a generalized holomorphic vector bundle over $M$ and $D: \Gamma(E)\to \Gamma(G^* M\otimes E)$ be a $T_{\mathbbm{C}} M$-connection valued in $G^*M\otimes E$
on $E$. Then $D$ is a generalized holomorphic connection  if and only if $D_X $, for any generalized holomorphic  vector field $X\in \Gamma(GM)$, sends generalized holomorphic sections of $E$ to generalized holomorphic sections of $E$.
\end{Lem}

\begin{proof}
If $D$ is a generalized holomorphic connection on $E$, by  Lemma \ref{pairingnew}, we see that $D_X s$ is a generalized holomorphic section of $E$ if $s$ is a generalized holomorphic section of $E$. Conversely, locally, write $D=d_{+}+A$, where $A$ is a matrix-valued $1$-form. Since $\overline{\partial}_E(D_X s)=0$ if $\overline{\partial}_E s=0$ and $\overline{\partial}_E X=0$, we must have $d_{-}\langle A, X\rangle =0$. This implies that $A$ is a matrix-valued generalized holomorphic $1$-form on $M$. So $Ds$ takes values in generalized holomorphic sections of $G^*M\otimes E$ when $s$ is a generalized holomorphic section.
\end{proof}

\begin{Ex}
When $M$ is a complex manifold, a generalized holomorphic connection on $E$ is a holomorphic connection; When $M$ is a symplectic manifold, since $G^*M$ is of rank $0$, a generalized holomorphic connection can only be zero. 
\end{Ex}

\subsection{Atiyah class in the $\rm{\check{C}}$ech cohomology}

Let us first recall the definition of $\rm{\check{C}}$ech cohomology. Let $\mathcal{F}$ be a sheaf on a manifold $M$. Fix an open covering $M=\cup_{i\in I} U_i$  with $I$ an ordered set. Denote by $U_{i_0\cdots i_k}:=U_{i_0}\cap \cdots \cap U_{i_k}$. Then set
\[C^k(\{U_i\},\mathcal{F}):=\Pi_{i_0<\cdots < i_k} \Gamma(U_{i_0\cdots i_k},\mathcal{F}).\]
There is a natural differential 
\[\mathbbm{d}: C^k(\{U_i\},\mathcal{F})\longrightarrow C^{k+1}(\{U_i\},\mathcal{F}),\qquad \alpha=\Pi \alpha_{i_0\cdots i_k}\mapsto \mathbbm{d}\alpha\]
with \[(\mathbbm{d}\alpha)_{i_0\cdots i_{k+1}}=\sum_{j=0}^{k+1}(-1)^k\alpha_{i_0\cdots \hat{i_j}\cdots i_{k+1}}|_{U_{i_0\cdots i_{k+1}}}.\]
One obtains a complex $(C^\bullet(\{U_i\},\mathcal{F}),\mathbbm{d})$ by the fact that $\mathbbm{d}^2=0$. The cohomology  of this complex is called the \emph{$\rm{\check{C}}$ech cohomology} of the sheaf $\mathcal{F}$ with respect to the fixed covering $M=\cup_{i} U_i$.

We shall give an obstruction class of the existence of generalized holomorphic connections by using of the first $\rm{\check{C}}$ech cohomology group. Let $E$ be a generalized holomorphic vector bundle over a generalized complex manifold $M$. Consider the ring $\mathcal{O}_M$  of generalized holomorphic functions and the sheaf of matrix-valued generalized holomorphic $1$-forms on $M$, or the sheaf of generalized holomorphic sections of $G^*M\otimes \End(E)$.

With respect to the trivialization $\varphi_i: \pi^{-1}(U_i)\to U_i\times \mathbbm{C}^r$ in Definition \ref{Definition3.1} on $E$, we may write a local generalized holomorphic connection on $U_i\times \mathbbm{C}^r$ in the form $d_++A_i$, where $A_i$ is a matrix-valued generalized holomorphic 1-form on $U_i$. They can be glued to a global connection on $E$ if and only if 
\[\varphi_i^{-1}\circ (d_++A_i)\circ \varphi_i=\varphi_j^{-1}\circ (d_++A_j)\circ \varphi_j\]
on $U_{ij}$, which is equivalently to
\begin{eqnarray*}
\varphi_j^{-1}\circ A_j\circ \varphi_j-\varphi_i^{-1}\circ A_i\circ \varphi_i&=&\varphi_j^{-1}\circ (\varphi_{ij}^{-1}\circ d_+\circ \varphi_{ij}-d_+)\circ \varphi_j\\ &=&\varphi_j^{-1}\circ (\varphi_{ij}^{-1}d_+ (\varphi_{ij}))\circ \varphi_j,
\end{eqnarray*}
where $\varphi_{ij}=\varphi_i\circ \varphi_j^{-1}$. By the relation $\varphi_{ij}\circ \varphi_{jk}\circ \varphi_{ki}=1$, the left hand side of the above equation is actually a $1$-coboundary in the $\rm \check{C}$ech cohomology. This suggests us to define a class which  measures the obstruction to a generalized holomorphic connection on a generalized holomorphic vector bundle. 

\begin{Def}
The \emph{ Atiyah class} $A(E)\in \mathrm{H}^1(M,G^*M\otimes \End(E))$
of a generalized holomorphic vector bundle $E$ on a generalized complex manifold $(M,\mathcal{J})$ is given by the $\rm \check{C}$ech cocycle
\[A(E)=\{U_{ij},\varphi_{j}^{-1}\circ (\varphi_{ij}^{-1}d_+ (\varphi_{ij}))\circ \varphi_j\},\]
where $d_+$ is the Lie algebroid differential of the $+i$-eigenbundle  $L_+$ of $\mathcal{J}$.
\end{Def}
\begin{Thm}
With notations above, 
 $E$ admits a generalized holomorphic connection if and only if  $A(E)\in \mathrm{H}^1(M,G^*M\otimes \End(E))$ vanishes.
\end{Thm}

\subsection{Atiyah class as an extension class}

Let $R$ be a commutative ring. For $R$-modules $M$ and $N$, let us recall the  extension groups $\mathrm{Ext}_R^n(M,N), n\geq 1$. Consider an extension  of $M$ by $N$ with length $n$, which is a long exact sequence
\[0\to N\to E_{n-1}\to \cdots\to E_0\to M\to 0 \]
of $R$-modules.
An \emph{elementary equivalence} of two extensions  of $M$ by $N$ with length $n$ is a commutative diagram:
\begin{equation*}\label{eq:ext1}
\CD
  0 @> >>  N @>  >>E_{n-1} @> >>\cdots @> >> E_0 @> >> M @> >> 0 \\
  @. @V \mathrm{Id} VV @V  VV @V  VV @V  VV@V \mathrm{Id} VV @. \\
 0 @> >>  N @>  >>E'_{n-1} @> >>\cdots @> >> E'_0 @> >> M @> >> 0,
 \endCD
\end{equation*}
where all the vertical maps $E_i\to E'_i$ for $i=0,\cdots,n-1$ are isomorphisms.  Elementary equivalence relations generate an equivalence relation. Then, for $n\geq 1$, $\mathrm{Ext}_R^n(M,N)$ is the Abelian group of equivalence classes of extensions of $M$ by $N$ with length $n$.

We shall use the first extension group to give an Atiyah class which describes the existence of a generalized holomorphic connection on a generalized holomorphic vector bundle over a regular generalized complex manifold $M$.

\begin{Def}
Let $E$ be a generalized holomorphic vector bundle over a regular generalized complex manifold $M$. The \emph{ Atiyah class} of $E$ is defined to be the first extension class 
\[A(E)\in \mathrm{Ext}_{\mathcal{O}_M}^1(E,G^*M\otimes E)\]
of the short exact sequence:
\[0\to G^*M\otimes E\xrightarrow{I} \mathfrak{J}^1 E\xrightarrow{\pi^1} E\to 0.\]
\end{Def}

\begin{Thm}
Let $E$ be a  generalized holomorphic vector bundle over a regular generalized complex manifold $M$. 
Then $E$ admits a generalized holomorphic connection if and only if  $A(E)=0$, namely, the above short exact sequence  splits.
\end{Thm}
\begin{proof}
If $S:E\to \mathfrak{J}^1 E$ is a splitting of (\ref{ex se}), then $D: \Gamma(E)\to \Gamma(G^*M\otimes E)$ defined by
\[D(\phi)=[\phi]-S(\phi)\]
is a generalized holomorphic connection on $E$, where $\phi$ is any local generalized holomorphic section of $E$.
Conversely, let $D$ be a generalized holomorphic connection on $E$. Define
\[S: E\to \mathfrak{J}^1 E,\qquad s_m\mapsto [\phi]_m-D(\phi)|_{m},\qquad \forall \phi\in \Gamma_m (E), s. t. \phi(m)=s_m.\]
First we check that $S$ is well-defined. It is equivalent to check $[\phi]_m-D(\phi)|_{m}=0_m$ when $\phi(m)=0_m$.
Under a local coordinate system $(U_i^1,\varphi_i;z,p,q,u^\alpha,u_\lambda^\alpha)$, we have $\phi=\sum_\alpha (u^\alpha\circ \phi) e_\alpha$, where $\{e_\alpha\}$ is a basis of local sections on $E$ given by $\varphi_i$. Then 
\[D(\phi)=\sum_\alpha (u^\alpha\circ \phi) D(e_\alpha)+d(u^\alpha\circ \phi)\otimes e_\alpha.\]
Restricting at $m$, we have
\[D(\phi)|_{m}=\sum_{\alpha,\lambda} \frac{\partial u^\alpha\circ \phi}{\partial z_\lambda} |_m dz_\lambda\otimes e_\alpha.\]
This implies that 
\[u_\lambda^\alpha([\phi]|_m-D(\phi)|_m)=\frac{\partial u^\alpha\circ \phi}{\partial z_\lambda} |_m-\frac{\partial u^\alpha\circ \phi}{\partial z_\lambda} |_m=0.\]
It is obvious that $u^\alpha([\phi]|_{m}-D(\phi)|_m)=0$. Therefore, we prove that $[\phi]_m-D(\phi)|_{m}=0_m$.
So $S$ is well-defined and it is complex linear restricting on each fiber.

Next, we shall check that $S$ is a homomorphism  between two generalized holomorphic vector bundles (see Definition \ref{homomorphism}). Let $\{e_\alpha\}$ be a basis of local generalized holomorphic sections on $E$. 
Suppose 
\[D(e_\alpha)=\sum_{\beta, \lambda} A_{\alpha\beta\lambda} dz_\lambda\otimes e_\beta,\]
where each $A_{\alpha\beta\lambda}:U_i\to \mathbbm{C}$ is a generalized holomorphic function. For $\phi\in \Gamma_m(E)$ and $\phi(m)=s_m$, we have
\[u_\lambda^\alpha([\phi]_m-D(\phi)_{|m})= -\sum_\beta (u^\beta \circ \phi)(m)A_{\beta \alpha \lambda}(m).\]
Since $u^\alpha([\phi]_m-D(\phi)|_{m})=u^\alpha\circ \phi(m)$ and noticing that $E|_{U_i}\cong U_i\times \mathbbm{C}^r$ and $\mathfrak{J}^1 E|_{U_i}\cong U_i\times \mathbbm{C}^{r+kr}$,
we have that locally $S$ can be expressed as 
\[S: U_i\times \mathbbm{C}^r \to U_i\times \mathbbm{C}^{r+rk},\qquad (z,p,q,w)\mapsto \big(z,p,q,w,-\sum_{\beta} w_\beta A_{\beta \alpha \lambda}(z,p,q)\big),\]
where $w=(w_1,\cdots,w_r)$ is the coordinate on $\mathbbm{C}^r$. This implies that $S$ is a homomorphism between generalized holomorphic vector bundles. 
It is easy to see $\pi^1\circ S=\mathrm{Id}$. Therefore, $S$ is a splitting of (\ref{ex se}).
\end{proof}

\subsection{Atiyah class in the Lie algebroid cohomology}

Let $A$ be a Lie algebroid and $(E,\nabla)$ an $A$-module  as defined in Definition \ref{3.6}.  Then we have a complex $(C^\bullet(A,E),d_E)$, where  $C^k(A,E)=\Gamma(\wedge^k A^*\otimes E)$ and the differential 
\[d_E: C^k(A,E)\longrightarrow C^{k+1}(A,E),\]
is given by
\begin{eqnarray*}
d_E\omega(a_1,\cdots,a_{k+1})&=&\sum_{i=1}^{k+1}(-1)^{i+1} \nabla_{a_i} \omega(a_1,\cdots,\widehat{a_i},\cdots,a_{k+1})\\ && \sum_{i< j} (-1)^{i+j}\omega([a_i,a_j],\cdots,\widehat{a_i},\cdots,\widehat{a_j},\cdots,a_{k+1}),
\end{eqnarray*}
for $a_i\in \Gamma(A)$.
The cohomology  of this complex is called the \emph{Lie algebroid cohomology}  of $A$ with coefficients in $E$, denoted by $\mathrm{H}^\bullet (A,E)$.

We have another way to introduce the Atiyah class in this cohomology.
A Lie pair $(P,A)$ is a Lie algebroid $P$ with a Lie subalgebroid $A$. 
Let us recall the construction of the Atiyah class of a Lie pair following \cite{CSX}. For a Lie pair $(P,A)$ and an $A$-module $(E,\nabla)$, 
\begin{itemize}
\item [\rm{(1)}] choose a $P$-connection $\hat{\nabla}: \Gamma(E)\to \Gamma(P^*\otimes E)$ which extends the flat $A$-connection $\nabla$ on $E$. 
\item [\rm{(2)}] The curvature $R^{\hat{\nabla}}\in \Gamma(\wedge^2 P^*\otimes \End(E))$ induces a tensor
\[R^{\hat{\nabla}}\in \Gamma(A^*\otimes (P/A)^*\otimes \End(E))\] as follows:
\[R^{\hat{\nabla}}(a,\tilde{u})e=\hat{\nabla}_a\hat{\nabla}_u e-\hat{\nabla}_u \hat{\nabla}_a e-\hat{\nabla}_{[a,u]} e,\qquad \forall a\in \Gamma(A),\tilde{u}\in \Gamma(P/A),e\in \Gamma(E).\]
\item[\rm{(3)}] $R^{\hat{\nabla}}$ is a $1$-cocycle in the cohomology of $A$ with coefficients in the $A$-module $A^\perp\otimes \End(E)$. Moreover, its cohomology class does not depend on the choice of the $P$-connection $\hat{\nabla}$ extending the flat $A$-connection $\nabla$. Write \[\alpha:=[R^{\hat{\nabla}}]\in \mathrm{H}^1(A,A^{\perp}\otimes \End(E)).\] \end{itemize}
This cohomology class $\alpha$ is called the \emph{ Atiyah class} of a Lie pair $(P,A)$ with respect to the $A$-module $E$. 
It vanishes if and only if there is an $A$-compatible $P$-connection on $E$; see \cite{CSX} for more details.

For our case,  suppose that $E$ is a generalized holomorphic vector bundle over a regular generalized complex manifold $M$. Then $E$ is an $L_{-}$-module and thus a $\rho(L_{-})$-module given by
\[\nabla_{\rho(l)} e:=\langle l, \overline{\partial}_E(e)\rangle,\qquad \forall l\in \Gamma(L_{-}), e\in \Gamma(E),\] 
where $\rho: T_{\mathbbm{C}} M\oplus T^*_{\mathbbm{C}} M\to T_{\mathbbm{C}} M$ is the projection.  
This is well-defined since for any $\xi\in \Gamma(L_{-})\cap \Gamma(T^*_{\mathbbm{C}} M)$, by definition, we have 
$\langle \xi,\overline{\partial}_E e\rangle =0$.
Consider the Lie pair $(T_{\mathbbm{C}} M,\rho(L_{-}))$.
Observe that \[(\rho(L_{-}))^\perp=L_{-}\cap T_{\mathbbm{C}}^* M=G^*M.\] We then obtain the Atiyah class  \[\alpha\in \mathrm{H}^1(\rho(L_{-}),G^*M\otimes \End(E))\]
of the Lie pair $(T_{\mathbbm{C}} M,\rho(L_{-}))$ with the $\rho(L_{-})$-module $E$.

\begin{Thm}
The Atiyah class for the above Lie pair vanishes if and only if there exists a generalized holomorphic connection on the generalized  holomorphic vector bundle $E$.
\end{Thm}
\begin{proof}
By  \cite[Proposition 2.3]{CSX} and \cite[Theorem 2.5]{CSX}, let 
\begin{eqnarray*} \mathcal{E}(U)&:=&\{e\in \Gamma(U,E),\nabla_a e=0,\forall a\in \Gamma(U,A)\},\\ 
\mathcal{P}(U)&:=&\{p\in \Gamma(U,P),[p,a]\subset \Gamma(U,A),\forall a\in \Gamma(U,A)\}.
\end{eqnarray*}
If the sheaf of sections of $E$ is isomorphic to $C^\infty(M)\otimes_{\mathbbm{R}} \mathcal{E}$ and the sheaf of sections of $P$ is isomorphic to $C^\infty(M)\otimes_{\mathbbm{R}} \mathcal{P}$, then
the Atiyah class vanishes if and only if $\nabla_{\mathcal{P}}\mathcal{E}\subset \mathcal{E}$.

 In our case, it is not hard to observe that $\mathcal{E}$ is the space of local generalized holomorphic sections of $E$ and 
and $\mathcal{P}$ is the space of generalized holomorphic vector fields  and sections of $\rho(L_-)$ on $M$. By Lemma \ref{perserve}, we obtain the result.
\end{proof}
\begin{Ex}
Let $E$ be a generalized holomorphic vector bundle over a holomorphic Poisson manifold $(M,\pi)$. The Atiyah class vanishes if and only if there exists a holomorphic connection $D$  on $E$ such that $D_{X}=0$ for any Hamiltonian vector field $X$ on $M$. 
In other words, the connection $1$-form takes values in the kernel of $\pi$.

This is different from the notion of the twisted Atiyah class of a holomorphic vector bundle $E$ on a holomorphic Poisson manifold $M$ introduced in \cite{CLX}, which vanishes if and only if there is a holomorphic $(T^*M)^{1,0}$-connection on $E$. Here $(T^*M)^{1,0}$ is the Lie algebroid associated with the holomorphic Poisson structure on $M$. 
\end{Ex}

\begin{Rm}
A recent work that draws our attention is the Atiyah class for a Manin pair \cite{BP}. A Manin pair is a Courant algebroid $P$ with a Dirac structure $A\subset P$.  Given a Manin pair $(P,A)$ and a predual vector bundle $B$ with an $A$-Dorfman action, the authors constructed an Atiyah class, which is in the first cohomology $\mathrm{H}^1(A,A^\perp\otimes \End(B))$. 
This Atiyah class vanishes if and only if there exists an $P$-compatible $A$-Dorfman connection on $B$.


For our case, we have a Manin pair $(T_{\mathbbm{C}} M\oplus T_{\mathbbm{C}}^*M, L_-)$ and a generalized holomorphic vector bundle $E$, which is an $L_-$-module.  Since $E$ is in general not a predual bundle of $L_-$, our case can not fall into the framework of \cite{BP}. The relationship between them is left to our future work.  It is desirable to find the method to define the Atiyah class for a Manin pair $(P,A)$ with an $A$-module $E$.
\end{Rm}

\begin{Rm}
For holomorphic vector bundles, the Atiyah classes, defined by using the $\rm{\check{C}}$ech cohomology, the extension class of the first-jet bundle and the Lie pair, are exactly the same by the isomorphisms of their respective coholomogy groups. See \cite{CSX,H}. For the generalized holomorphic vector bundles, the three methods to define the Atiyah class are introduced. The relations of the three cohomology groups are not clear, which will be studied in the future.

\end{Rm}


\newpage
\begin{thebibliography}{10}
\bibitem{Atiyah}
M. F. Atiyah, Complex analytic connections in fibre bundles,  {\it Trans. Amer. Math. Soc. } 85 (1957), 181-207.

\bibitem{AB}
M. Abouzaid and M. Boyarchenko, Local structures of generalized complex manifolds, {\it J. Symplectic Geom. } 4 (2006), no. 1, 43-62.

\bibitem{Bailey}
M. Bailey, Local classification of generalized complex structures, {\it J. Differential Geom.} 95 (2013), 1-37.
\bibitem{BP}
P. Batakidis and F. Petalidou, Atiyah class of a Manin pair, arXiv: 2002.10175.




\bibitem{CLangX}
Z. Chen, H. Lang, and M. Xiang, Atiyah classes of strongly homotopy Lie pairs, Algebra Colloq. 26 (2019),
no. 2, 195-230.
\bibitem{CLX} Z. Chen, Z. Liu and M. Xiang, Kapranov's construction of SH Leibniz algebras, {\it Homology Homotopy Appl.} 22 (2020), no. 1, 141-165.

\bibitem{CSX} Z. Chen, M. Stienon and P. Xu, From Atiyah classes to homotopy Leibniz algebras, {\it Comm. Math. Phys. } 341 (2016), no. 1, 309-349.




\bibitem{Crainic}
M. Crainic, Generalized complex structures and Lie brackets, {\it Bulletin of the Brazilian Mathematical Society, New Series,} 42 (2011), no. 4, 559-578.
\bibitem{G} V. L. Ginzburg, Grothendieck groups of Poisson vector bundles, {\it J. Symplectic Geom.} 1 (2001), 121-169.
\bibitem{Gualtieri}
M. Gualtieri, Generalized complex geometry, {\it  Ann. of Math. (2)} 174 (2011), no. 1, 75-123.

\bibitem{Hitchin}
N. Hitchin, Generalized Calabi-Yau manifolds, {\it Q. J. Math.} 54 (2003), 281-308.
\bibitem{Hitchin3}
N. Hitchin, Generalized holomorphic bundles and the B-field action, {\it J. Geom. Phys. } 61 (2011), no. 1, 352-362.
\bibitem{Hitchin4}
N. Hitchin, Poisson modules and generalized geometry, 
Lizhen Ji (Ed.), Geometry and Analysis Vol 1, Advanced Lectures in Mathematics, vol. 17, Higher Education Press, Beijing (2010), pp. 403-406.
\bibitem{Hong}
W. Hong, Atiyah classes of Lie bialgebras, {\it J. Lie Theory} 29 (2019), no. 1, 263-275.

\bibitem{H}
D. Huybrecht, Complex Geometry, An Introduction, Springer, Berlin, 2005.

\bibitem{Kap99} 
M. Kapranov,  Rozansky-Witten invariants via Atiyah classes, {\it Compositio Math.} 115 (1999), no. 1, 71-113.

\bibitem{K03}
M. Kontsevich, Deformation quantization of Poisson manifolds, {\it Lett. Math. Phys.} 66 (2003), no. 3, 157-216.



\bibitem{LSX}
C. Laurent-Gengoux, M. Stienon and P. Xu, Holomorphic Poisson manifolds and holomorphic Lie algebroids, {\it Int. Math. Res. Notice} 2008 (2007), rnn088, 46 pages.

\bibitem{MSX}
R. Mehta, M. Stienon and P. Xu, The Atiyah class of a dg-vector bundle, {\it C. R. Math. Acad. Sci. Paris}
353 (2015), no. 4, 357-362.
\bibitem{LWX}
Z. Liu, A. Weinstein and P. Xu, Manin triples for Lie bialgebroids, {\it J. Differential Geom.} 45 (1997), no. 3, 547-574.

\bibitem{Mo71}
P. Molino, Classe d'Atiyah d'un feuilletage et connexions transverses
   projetables, {\it  C. R. Acad. Sci. Paris} Ser. A-B 272 (1971), A779-A781 (French).



\bibitem{N}
H. Nguyen-van, Relations entre les diverses obstructions relatives \`a
   l'existence d'une connexion lin\'eaire invariante sur un expace
   homog\`ene, {\it C. R. Acad. Sci. Paris} 260 (1965), 45-48 (French).


\bibitem{Ornea}
L. Ornea and R. Pantilie, On holomorphic maps and generalized complex geometry, {\it J. Geom. Phys.} 61 (2011), no. 8, 1502-1515.
\bibitem{Park}
J. S. Park, Topological open $p$-branes,  Proceedings of the KIAS conference
Symplectic Geometry and Mirror Symmetry, pp. 311-384, ed. by Fukaya, Oh,
Ono, Tian. World Scientific, 2001. 




\bibitem{R}
S. Rayan, Geometry of co-Higgs bundles, Ph.D thesis, Oxford University, 2011.
\bibitem{W1}
Y. C. Wang, Generalized holomorphic structures, {\it J. Geom. Phys.} 61 (2011), 1976-1984.




\end{thebibliography}
\end{document}